\documentclass[a4paper,11pt]{article}
\usepackage{amsfonts}
\usepackage{amsmath}
\usepackage{amssymb}
\usepackage{amsthm}
\usepackage{graphicx}
\usepackage{enumitem}
\usepackage{comment}
\usepackage{tikz-cd}
\usepackage{scalerel}
\usepackage{booktabs}
\usepackage[utf8]{inputenc}
\usepackage[toc]{glossaries}

\usepackage{tikz}
\usetikzlibrary{trees,arrows,automata,positioning}
\tikzstyle{every picture}=[scale=.5,inner sep=10]

\usepackage{hyperref}

\usepackage{biblatex}
\theoremstyle{plain}
\newtheorem{theorem}{Theorem}[section]
\newtheorem{proposition}[theorem]{Proposition}
\newtheorem{lemma}[theorem]{Lemma}
\newtheorem{corollary}[theorem]{Corollary}

\theoremstyle{definition}
\newtheorem{definition}[theorem]{Definition}
\newtheorem{example}[theorem]{Example}
\newtheorem{remark}[theorem]{Remark}
\newtheorem{notation}[theorem]{Notation}

\DeclareMathOperator{\ev}{\mathrm{ev}}

\DeclareMathOperator{\Mstack}{\mathcal{M}}
\DeclareMathOperator{\Mbstack}{\overline{\mathcal{M}}}

\makeglossaries

\newglossaryentry{X}{
	name=$X$,
	description=the target variety
}
\newglossaryentry{tis}{
	name={$T_0,\ldots T_k$},
	description=homogeneous basis of $H^*(X)$
}
\newglossaryentry{gef}{
	name=$g^{ef}$,
	description=inverse of intersection product
}
\newglossaryentry{modulicurves}{
	name=$\Mbstack_{g,n}$,
	description=moduli space of stable curves
}
\newglossaryentry{pii}{
	name=$\pi_i$,
	description=forgetful map forgetting the $i$-th point
}
\newglossaryentry{q}{
	name=$q$,
	description=gluing map that glues two curves together
}
\newglossaryentry{delta}{
	name=$\delta$,
	description=gluing map that glues a curve to itself
}
\newglossaryentry{psi}{
	name=$\psi$,
	description=a tautological class
}
\newglossaryentry{kappa}{
	name=$\kappa$,
	description=a tautological class
}
\newglossaryentry{modulimaps}{
	name={$\Mbstack_{g,n}(X,\beta)$},
	description=moduli space of stable maps
}
\newglossaryentry{ev}{
	name=$\ev_i$,
	description=evaluation map at the $i$-th point of a stable map
}
\newglossaryentry{tildepsi}{
	name=$\tilde\psi$,
	description=cocycle on the moduli space of stable maps
}
\newglossaryentry{virfund}{
	name={$[\Mbstack_{g,n}(X,\beta)]^\text{vir}$},
	description=virtual fundamental class
}
\newglossaryentry{gwi}{
	name=$<\tau_{a_1}(\gamma_1)\cdots\tau_{a_n}(\gamma_n)>_{g,\beta}^X$,
	description=Gromov-Witten invariant
}
\newglossaryentry{symbolgwi}{
	name={$<\psi^{a_1}\gamma_1,\ldots,\psi^{a_n}\gamma_n>$},
	description=Gromov-Witten invariant using symbol notation
}
\newglossaryentry{tautring}{
	name={$R^*(\Mbstack_{g,n})$},
	description=tautological ring of the moduli space of curves
}
\newglossaryentry{gamma}{
	name={$\Gamma$},
	description=stable graph
}
\newglossaryentry{Mgamma}{
	name={$\Mbstack_\Gamma$},
	description=product of Moduli spaces associated to $\Gamma$
}
\newglossaryentry{xigamma}{
	name={$\xi_\Gamma$},
	description=glueing map associated to $\Gamma$
}
\newglossaryentry{decgraph}{
	name={$\Gamma_\theta$},
	description=decorated stable graph
}
\newglossaryentry{Msmooth}{
	name={$\mathcal{M}_{g,n}$},
	description=moduli space of smooth curves
}
\newglossaryentry{Mrt}{
	name={$\mathcal{M}^\mathrm{rt}_{g,n}$},
	description=moduli space of curves with rational tail
}
\newglossaryentry{Mct}{
	name={$\mathcal{M}^\mathrm{ct}_{g,n}$},
	description=moduli space of curves of compact type
}
\newglossaryentry{strataalgebra}{
	name={$S_{g,n}$},
	description=the strata algebra
}
\newglossaryentry{relations}{
	name={$R_{g,n}$},
	description=algebra of relations
}
\newglossaryentry{newrels}{
	name={$R_{g,n}^\mathrm{new}$},
	description=new relations
}
\newglossaryentry{pixtonrels}{
	name={$P_{g,n}$},
	description=algebra of Pixton's relations
}
\newglossaryentry{P}{
	name={$\mathbb{Q}[GW(X)]$},
	description=polynomial ring in formal Gromov-Witten invariants
}
\newglossaryentry{G}{
	name={$G_{g,\beta}^X$},
	description=linear combinations of formal Gromov-Witten invariants
}
\newglossaryentry{eta}{
	name={$\eta$},
	description=evaluation map for formal Gromov-Witten invariants
}
\newglossaryentry{Sprim}{
	name={$S_\mathrm{prim}$},
	description=the primitive part of a relation $S$
}
\newglossaryentry{sigma}{
	name={$\Sigma(S)(\gamma_1,\ldots,\gamma_n)$},
	description=the symbol map
}
\newglossaryentry{equiv}{
	name={$\sim$},
	description=equivalence relation on $G_{g,\beta}^X$
}
\newglossaryentry{Theta}{
	name={$\Theta_{j,k}(a,b)$},
	description=an element of $G_{g,\beta}^X$
}
\newglossaryentry{Phi}{
	name={$\Phi_k(\gamma_1,\gamma_2,\gamma_3)$},
	description=system of equations in $G_{g,\beta}^X$
}
\newglossaryentry{Xr}{
	name={$X_r$},
	description=the blowup of $\mathbb{P}^2$ in $r$ points
}
\newglossaryentry{NHPK}{
	name={$N^{(g)}_{d,\alpha}$,$H^{(g)}_{d,\alpha}$,$P^{(g)}_{d,\alpha}$,$K^{(g)}_{d,\alpha}$},
	user1=,
	description=Gromov-Witten invariants of $X_r$
}
\newglossaryentry{equivxr}{
	name={$\cong$},
	description=equivalence relation used for computing Gromov-Witten invariants of $X_r$
}

\author{Thomas Wennink}
\title{Reconstruction theorems for genus $2$ Gromov-Witten invariants}
\date{}
\begin{document}
\maketitle

\begin{abstract}
\noindent
We use Pixton's relations to prove a reconstruction theorem for genus~$2$ Gromov-Witten invariants in the style of Kontsevich-Manin (genus~$0$) and Getzler (genus~$1$).
We also calculate genus~$2$ (descendant) Gromov-Witten invariants of $\mathbb{P}^2$ blown up at a finite number of points in general position.
\end{abstract}

\section{Introduction}
In the early days of Gromov-Witten theory Kontsevich proved a recursive formula that solved the long-standing problem of counting rational curves in $\mathbb{P}^2$ of degree $d$ that go through $3d-1$ points in general position.

This formula was proven using the WDVV relation in the tautological ring $R^1(\Mbstack_{0,4})$.
Given such a tautological relation, we obtain a relation in the cohomology ring of the moduli space of stable maps by pulling it back along the forgetful map that remembers and stabilizes the source curve only.
The Splitting Lemma then expresses these pulled back relations in terms of Gromov-Witten invariants of $\mathbb{P}^2$.
This approach can be used to prove reconstruction theorems for more general smooth projective target varieties $X$.

The first reconstruction theorem by Kontsevich-Manin \cite{kontmaninrecon} states that when the cohomology of $X$ is generated by divisors, the genus~$0$ Gromov-Witten invariants can be computed using recursive formulas from the $3$-pointed invariants as initial values.

In \cite{getzlerg1} Getzler discovered a new tautological relation in genus~$1$ and used it to prove a similar reconstruction theorem:
For target varieties $X$ whose primitive cohomology is in $H^{\leq 2}(X)$, all genus $1$ invariants can be computed from the $1$-pointed genus~$1$ invariants and all genus~$0$ invariants.

Later Belorousski and Pandharipande \cite{belpandg2} found a new tautological relation in genus~$2$.
They used this relation to find a recursive formula for genus~$2$ invariants of $\mathbb{P}^2$.
But they were unable to prove a general reconstruction theorem for genus~$2$ invariants in the style of Kontsevich-Manin and Getzler.

In \cite{liug2semisimple} Liu expresses all genus~$2$ Gromov-Witten invariants of projective varieties with semisimple quantum cohomology in terms of genus~$0$ and~$1$ invariants.
However there are many varieties that are not semisimple yet satisfy Getzler's hypothesis.

Our main result is a reconstruction theorem for genus~$2$ invariants:
\newtheorem*{g2rec}{Theorem \ref{g2reconstruction}}
\begin{g2rec}
	\textit{If $P^i(X)=0$ for $i>2$, then all (including descendant) genus two Gromov-Witten invariants can be reconstructed recursively from genus two invariants with at most two points and invariants of lower genus.}
\end{g2rec}
We prove this using 3 tautological relations that we obtain from the Pixton relations.

Our approach is similar to the approach of Getzler.
In \cite[Definition 3.1]{getzlerg1} Getzler defines the \emph{symbol} of a tautological relation by first pulling that relation back along the forgetful map to obtain a relation between Gromov-Witten invariants, and then setting most of it terms to zero.
There is a total order on Gromov-Witten invariants such that the symbol consists exactly of those terms that are of maximal order.
So when one solves a system of equations obtained from taking symbols, one can then express the solved invariants in terms of lower order invariants.

We will describe a formal framework for this method.
As part of this we extend Getzler's notion of a symbol to the case where there are $\psi$-classes (see Definition \ref{symboldef}), as in genus~$2$ we can no longer eliminate the $\psi$-classes.

The fact that we can reduce to one point in genus~$1$ and two points in genus~$2$ leads one to speculate if we could use the Pixton relations to reconstruct genus $g$ Gromov-Witten invariants from $g$-pointed ones.
We do not have much evidence in this direction besides the fact that for all $g$ there seems be a \emph{new} relation in $R^{g+1}(\Mbstack_{g,2g+2})$ for which the part that contributes to the symbol has a simple form (it contains no $\psi$-classes). Pixton (\cite{pixtonunpublished}) has proven that there is always a new relation in $R^{g+1}(\Mbstack_{g,2g+2})$, but his formula for the part that contributes to the symbol is conjectural.

While we believe that the main significance of our result is theoretical, we have tried to apply our theorem to a simple yet nontrivial example.
However the requirement to know all $2$-pointed genus~$2$ Gromov-Witten invariants turns out to be quite demanding.
By the hard Lefschetz theorem all K\"ahler surfaces have no primitive cohomology after degree two.
But most of these spaces have infinitely many invariants with at most 2 points.
Finding a strategy that calculates invariants with at most 2 points might not be much easier than finding a strategy that calculates all invariants.

So applying the theorem for spaces where the Gromov-Witten invariants are yet unkown appears to be quite difficult.
But one can still use the algorithm prescribed by the theorem for a reconstruction.

In this way we use one of the relations prescribed by the theorem to construct a new algorithm specific to $X_r$, the blowup of $\mathbb{P}^2$ at $r$ points in general position.
\newtheorem*{p2blowup}{Theorem \ref{p2blowupcomputation}}
\begin{p2blowup}
	\textit{The algorithm described in Sections \ref{dpg0rec}, \ref{dpg1rec}, and \ref{dpg2rec} reconstructs all genus~$0$,$1$, and $2$ Gromov-Witten invariants with descendants of $X_r$ from the finitely many initial cases in Lemma~\ref{DPbasecases} and $<pt^2>^{X_r}_{0,H}=1$.}
\end{p2blowup}
The genus~$0$ case has already been done by G\"ottsche and Pandharipande in \cite{gotpand}.
We have added calculations for genus~$1$ and~$2$.
Our approach is an extension of theirs and of the computation of the genus~$2$ invariants of $\mathbb{P}^2$ by Belorousski and Pandharipande in \cite{belpandg2}.
When $r\leq8$, $X_r$ is a del Pezzo surface.
The \emph{primary} Gromov-Witten invariants of del Pezzo surfaces are enumerative and algorithms to compute them in any genus are found in \cite{vakilcountingcurves}, \cite{shovalshustin}, and \cite{brugalle}.
In the unpublished paper \cite{parkerP2blowups} Parker describes a method to calculate primary Gromov-Witten invariants of $X_r$ in any genus.
Our algorithm seems to be the only existing one that computes the full theory, including descendant invariants, in genus~$2$.
(Note that Liu's reconstruction theorem in \cite{liug2semisimple} could also be applied to $X_r$ as it is semisimple.)
We have written a computer program that implements our algorithm and the results agree with those stated in the literature for del Pezzo surfaces.
This computer program, together with a program we use to find the relations we need for our reconstruction theorem, can be found online at
	\text{\url{https://github.com/Wennink/gwreconstruction}}.
Both programs use the \verb|admcycles| project \cite{admcycles}.
During our search for relations we expanded upon Pixton's code to improve the calculation time and memory usage when calculating Pixton's relations.
This part of our program has been incorporated into the \verb|admcycles| project.

\subsection*{Acknowledgements}
I thank my PhD supervisor Nicola Pagani for giving me this problem to work on and for all his help and support.
I would also like to thank Aaron Pixton for helpful comments, in particular regarding the relation in $R_{g,2g+2}^{g+1}$.

\section{Conventions}
Throughout $X$ will be a smooth complex projective variety.
When we talk about cohomology classes on $X$ we implicitly take them to be homogeneous.
Unless otherwise specified all cohomology rings will be with rational coefficients.
When we say genus we always mean the arithmetic genus (unless otherwise specified).
We will make use of intersection theory as described in the book of Fulton \cite{fultonintersectiontheory} and its extension to the cohomology of proper Deligne-Mumford stacks \cite[Section 2]{abvgwidms}.

\section{Gromov-Witten theory}
In this section we go over the fundamental results in Gromov-Witten theory that we need for our reconstruction.

For $2g-2+n>0$, let $\Mbstack_{g,n}$ be the moduli space of $n$-pointed stable curves of genus~$g$.
An introduction to the moduli space of curves can be found in \cite{harrismorrison} and \cite{gac2}.
We have the forgetful maps $\pi_i:\Mbstack_{g,n+1}\rightarrow\Mbstack_{g,n}$ that forget the $i$-th point and stabilize the curve.
There are also gluing maps
\begin{align*}
	q :  \Mbstack_{g_1,n_1+1}\times\Mbstack_{g_2,n_2+1} & \rightarrow\Mbstack_{g_1+g_2,n_1+n_2},\\
	\delta :  \Mbstack_{g,n+2} & \rightarrow\Mbstack_{g+1,n},
\end{align*}
where \gls{q} glues the last marked point of each of the two curves together and \gls{delta} glues the last two markings of the same curve together.

The universal family is given by the forgetful map
\[\begin{tikzcd}
	\Mbstack_{g,n+1}\arrow{d}{\pi_{n+1}} \\
	\Mbstack_{g,n} \arrow[u,"\sigma_1\ldots\sigma_n", bend left].
\end{tikzcd}\]
Here $\sigma_1,\ldots,\sigma_n$ are the sections.
Let $\omega_{\pi_{n+1}}$ be the relative dualizing sheaf.
We define the \emph{\gls{psi}-classes} $\psi_i$ for $1\leq i\leq n$ as
\[\psi_i:=c_1(\sigma_i^*\omega_{\pi_{n+1}})\in H^2(\Mbstack_{g,n}),\]
where $c_1$ is the first Chern class.
By forgetting the point corresponding to the $\psi$-class we obtain the \emph{\gls{kappa}-classes}
\[\kappa_i:=\pi_{j\,*}(\psi^{i+1}_j)\in H^2(\Mbstack_{g,n}),\]
for $i\geq0$.

For any $g,n>0$, $X$ a smooth projective variety and $\beta\in H_2(X)$, let $\Mbstack_{g,n}(X,\beta)$ be the moduli space of $n$-pointed stable maps of genus~$g$. A construction of the moduli space of stable maps can be found in Chapter~\RN{5} of \cite{maninbookfrobenius}.
It is a proper Deligne-Mumford stack, but it may be singular and fail to be equidimensional, in general.
We have a virtual fundamental class
\[\text{\gls{virfund}}\in A_{k}(\Mbstack_{g,n}(X,\beta)),\]
where
\[k=\int_\beta c_1(T_X)+(\text{dim}(X)-3)(1-g)+n\]
is the \emph{expected dimension} of $\Mbstack_{g,n}(X,\beta)$. (See\cite{behrendfantechi} for the construction of the virtual fundamental class and its properties.)

Similarly to the moduli space of stable curves we have a universal family
\[\begin{tikzcd}
	\Mbstack_{g,n+1}(X,\beta) \arrow{r}{\ev_{n+1}} \arrow{d}{\pi_{n+1}} & X \\
	\Mbstack_{g,n}(X,\beta) \arrow[u,"\sigma_1\ldots\sigma_n", bend left]&
\end{tikzcd}\]
and \emph{\gls{tildepsi}-classes}
\[\tilde\psi_i:=c_1(\sigma_i^*\omega_{\pi_{n+1}})\in H^2(\Mbstack_{g,n}(X,\beta)).\]

\begin{definition}
Let $\beta\in H_2(X)$ and $\gamma_1,\ldots,\gamma_n\in H^*(X)$.
We define a \emph{Gromov-Witten invariant} of $X$ to be
\[\text{\gls{gwi}}:=\int_{[\Mbstack_{g,n}(X,\beta)]^\text{vir}}\tilde\psi_1^{a_1}\cup\cdots\cup\tilde\psi_n^{a_n}\cup\ev^*(\gamma_1\otimes\cdots\otimes\gamma_n).\]

If $a_i$ is zero then we may simply write $\gamma_i$ in place of $\tau_0(\gamma_i)$.
Later we will fix $X$ and leave out $X$ in the notation.
Invariants without $\psi$-classes are called \emph{primary invariants} while invariants with $\psi$-classes are called \emph{descendant invariants}.
\end{definition}
Note that for a Gromov-Witten invariant to be nonzero, $\beta$ needs to be effective.

Gromov-Witten invariants inherit the supercommutativity of the cohomology of $X$, i.e.
\[\left<\cdots\tau_{k_i}(\gamma_i)\tau_{k_{i+1}}(\gamma_{i+1})\cdots\right>_{g,\beta}^X=
(-1)^{|\gamma_i||\gamma_{i+1}|}\left<\cdots\tau_{k_{i+1}}(\gamma_{i+1})\tau_{k_i}(\gamma_i)\cdots\right>_{g,\beta}^X.\]

\subsection{The tautological ring}
We will obtain a relation in the cohomology ring of the moduli space of stable maps by pulling back relations from the tautological ring of the moduli space of stable curves.
The tautological ring is a well studied subring of the cohomology ring that contains most classes arising from geometric constructions.

\begin{definition}
	The system of \emph{tautological rings} of the moduli space of stable curves is the minimal system of $\mathbb{Q}$-subalgebras $\text{\gls{tautring}} \subset \{H^{2*}(\Mbstack_{g,n},\mathbb{Q})\}$ that is closed under pushforward along all the natural gluing and forgetful maps.
\end{definition}

\begin{remark}
	The tautological ring can also be defined as a system of $\mathbb{Q}$-subalgebras of the Chow rings $A^{*}(\Mbstack_{g,n},\mathbb{Q})$.
	The image of the Chow tautological ring under the cycle map is the tautological ring in cohomology.
	Everything we state about the tautological ring of the moduli space of curves holds in both Chow and cohomology.
\end{remark}

From an $n$-pointed stable curve we obtain an $n$\emph{-pointed stable graph} by taking the irreducible components as vertices and nodes as edges.
Marked points become legs.
The two half edges that make up an edge correspond to the inverse images of a node under the normalization map of the curve.
The genus of a vertex is the geometric genus of the corresponding irreducible component.

An $n$-pointed stable graph $\Gamma$ prescribes a gluing of moduli spaces
\[\xi_\Gamma:=\prod_{v\in V}\Mbstack_{g(v),n(v)}\rightarrow\Mbstack_{g,n},\]
where $V$ is the set of vertices of $\Gamma$.
We call $g=h^1(\Gamma)+\sum_{v\in V} g(v)$ the genus of $\Gamma$.
(The first Betti number $h^1$ of a graph is the number of independent loops.)

Let $\theta$ be a monomial on $\prod_{v\in V}\Mbstack_{g(v),n(v)}$ that is a product of $psi$- and $\kappa$-classes pulled back along the projection maps.
The stable graph $\Gamma$ together with $\theta$ forms \emph{decorated stable graph} \gls{decgraph}.
It has a corresponding \emph{decorated stratum class}
\[[\Gamma_\theta]:=\frac{1}{|\text{Aut}(\Gamma)|}{\xi_\Gamma}_*(\theta)\in H^*(\Mbstack_{g,n}).\]
	We write $[\Gamma]:=[\Gamma_1]$ for a decorated stratum class with trivial decoration.
When we draw a decorated graph we write the $\psi$-classes at the corresponding half-edges and the $\kappa$-classes at corresponding vertices.
If we draw a decorated graph without specifying a numbering on $n'$ legs, it corresponds to the $\mathfrak{S}_n'$-invariant sum over all possible ways to number the legs, divided by the size of the $\mathfrak{S}_n'$-orbit.
For example
\[
\left[	\begin{tikzpicture}[baseline,el/.style = {inner sep=2pt, align=left, sloped},every child node/.style={inner sep=1,font=\tiny}]
      \tikzstyle{level 1}=[counterclockwise from=0,level distance=9mm,sibling angle=120]
			\node (A0) [draw,circle,inner sep=1] at (0:1) {$\scriptstyle{4}$} child {node {}};
      \tikzstyle{level 1}=[counterclockwise from=120,level distance=9mm,sibling angle=120]
      \node (A1) [draw,circle,inner sep=1] at (180:1) {$\scriptstyle{0}$} child {node {1}} child {node {}};

			\path (A0) edge (A1);
	\end{tikzpicture}\right]
	=
	\frac{1}{2}\left[	\begin{tikzpicture}[baseline,el/.style = {inner sep=2pt, align=left, sloped},every child node/.style={inner sep=1,font=\tiny}]
      \tikzstyle{level 1}=[counterclockwise from=0,level distance=9mm,sibling angle=120]
			\node (A0) [draw,circle,inner sep=1] at (0:1) {$\scriptstyle{4}$} child {node {2}};
      \tikzstyle{level 1}=[counterclockwise from=120,level distance=9mm,sibling angle=120]
      \node (A1) [draw,circle,inner sep=1] at (180:1) {$\scriptstyle{0}$} child {node {1}} child {node {3}};

			\path (A0) edge [bend left=0.000000] (A1);
	\end{tikzpicture}\right]
	+\frac{1}{2}\left[	\begin{tikzpicture}[baseline,el/.style = {inner sep=2pt, align=left, sloped},every child node/.style={inner sep=1,font=\tiny}]
      \tikzstyle{level 1}=[counterclockwise from=0,level distance=9mm,sibling angle=120]
			\node (A0) [draw,circle,inner sep=1] at (0:1) {$\scriptstyle{4}$} child {node {3}};
      \tikzstyle{level 1}=[counterclockwise from=120,level distance=9mm,sibling angle=120]
      \node (A1) [draw,circle,inner sep=1] at (180:1) {$\scriptstyle{0}$} child {node {1}} child {node {2}};

			\path (A0) edge [bend left=0.000000] (A1);
	\end{tikzpicture}\right]
.\]

\begin{theorem} [Proposition 11 in \cite{grabpandnontaut}]
	The tautological ring \gls{tautring} is generated additively by the decorated strata class $[\Gamma_\theta]$ for all $n$-pointed genus $g$ stable graphs $\Gamma$ and all monomials $\theta$ in $\kappa$- and $\psi$-classes in $H^*(\Mbstack_\Gamma)$.
\end{theorem}

Whenever we pullback a decorated stratum class along the forgetful map we use the following Lemma.
\begin{lemma}[Lemma~17.4.28 in \cite{gac2}]
	\label{pullbackcurves}
	\begin{enumerate}[label=\roman*)]
	\item $\pi_j^*(\kappa_i)=\kappa_i-\psi_j^i$,
	\item $\pi_j^*(\psi_i)=\psi_i-\Bigl[\begin{tikzpicture}[baseline,el/.style = {inner sep=2pt, align=left, sloped},every child node/.style={inner sep=1,font=\tiny}]
      \tikzstyle{level 1}=[counterclockwise from=315,level distance=9mm,sibling angle=90]
			\node (A0) [draw,circle,inner sep=1] at (0:1.3) {$\scriptstyle{0}$} child {node {j}} child {node {i}};
      \tikzstyle{level 1}=[counterclockwise from=135,level distance=9mm,sibling angle=90]
			\node (A1) [draw,circle,inner sep=1] at (0:0) {$\scriptstyle{g}$} child {node {}} child {node {}};

      \path (A0) edge (A1);
	\path    (-.7,-0.9) -- (-.7,1) node [midway,sloped] {$\dots$};
	\end{tikzpicture}\;\;\Bigr]$,
		\item $\pi_j^*([\Gamma])=\sum_{v\in\Gamma}[\Gamma_v]$, where $[\Gamma_v]$ is the graph obtained by adding the $j$-th leg to the vertex $v$.
	\end{enumerate}
\end{lemma}

We can avoid the use of $\psi$-classes in genus~$0$ and~$1$, while we only need singlular $\psi$-classes in genus~$2$:
\begin{proposition} [Proposition 2.5 in \cite{halfspin2proportionalities}]
	\label{psivanishing}
Any monomial of $\psi$-classes of degree at least $\max(g, 1)$ in $R^*(\Mbstack_{g,n})$ can be
expressed in terms of the boundary strata classes that involve no $\kappa$-classes, that is, in terms of the dual
graphs with at least one edge, decorated only by $\psi$-classes.
\end{proposition}

\subsection{Tautological relations}
The generators of the tautological ring form the strata algebra.
\begin{definition}
	Define the \emph{strata algebra} \gls{strataalgebra} to be the free $\mathbb{Q}$-algebra generated by the $n$-pointed genus $g$ decorated strata classes $[\Gamma_\theta]$, for which the degree of $\theta$ at the vertex $v$ does not exceed the dimension $3g(v)-3+n(v)$ of the moduli space at $v$.
	The multiplication on $S_{g,n}$ is the one inherited from the intersection theory of $\Mbstack_{g,n}$, as described in \cite[Equation~(11)]{grabpandnontaut}.
\end{definition}
We define \gls{relations}, the $\mathbb{Q}$-algebra of \emph{tautological relations}, by the short exact sequence
\[0\rightarrow R_{g,n}\rightarrow S_{g,n} \rightarrow R^*(\Mbstack_{g,n})\rightarrow 0.\]
As expected we denote the restriction to each degree $r\in\mathbb{Z}_{\geq0}$ by
\[0\rightarrow R_{g,n}^r\rightarrow S_{g,n}^r \rightarrow R^r(\Mbstack_{g,n})\rightarrow 0.\]
The Pixton relations from \cite{relationsvia3spin} and \cite{jandapixtonrelschow} form a subalgebra $\text{\gls{pixtonrels}}\subseteq R_{g,n}$.
It is conjectured that $P_{g,n}=R_{g,n}$.

\begin{definition}
	\label{gwpbdefdef}
	For every $\beta\in H_2(X)$ we define a map
	\[<.\, ;\ldots >_\beta:S_{g,n}\otimes \left(H^*(X)\right)^{\otimes n}\rightarrow \mathbb{Q},\]
	by taking
\begin{equation}
	\label{gwpbdef}
	<S;\gamma_1,\ldots,\gamma_n>_\beta:=\int_{[\Mbstack_{g,n}(X,\beta)]^\text{vir}}F^*(p(S))\cup\ev^*(\gamma_1\otimes\cdots\otimes\gamma_n),
\end{equation}
where $p$ is the map $S_{g,n}\twoheadrightarrow R^*(\Mbstack_{g,n})$ and $F:\Mbstack_{g,n}(X,\beta)\rightarrow \Mbstack_{g,n}$ is the forgetful map that only remembers and stabilizes the source curve.
\end{definition}

\subsection{The Splitting Lemma}
We can use the Splitting Lemma to write the right-hand side of (\ref{gwpbdef}) as a polynomial in Gromov-Witten invariants.
The statements of the Splitting Lemma we found in the literature were either more abstract or less general, so we have chosen to give this explicit description here.

Let $\chi_\Gamma$ be the pullback of $\xi_\Gamma$ along the forgetful map $F$ that only remembers and stabilizes the source curve.
We have the fiber square
\[\begin{tikzcd} 
F^*(\Mbstack_\Gamma) \arrow[r, "\chi_\Gamma"] \arrow[d,"F_\Gamma"] & \Mbstack_{g,n}(X,\beta) \arrow[d,"F"] \\
\Mbstack_\Gamma \arrow[r,"\xi_\Gamma"] & \Mbstack_{g,n},
\end{tikzcd}\]
where
\[F_\Gamma^*(\Mbstack_\Gamma)=\bigoplus_{\sum_v\beta(v)=\beta}\prod_{v\in V}\Mbstack_{g(v),n(v)}(X,\beta(v))\]
is a direct sum over all the ways to divide the degree $\beta$ among the vertices.
Since flat pullback commutes with proper pushforward on fiber squares we have
\[F^*[\Gamma_\theta]=\frac{1}{|\text{Aut}(\Gamma)|}{\chi_\Gamma}_*F_\Gamma^*(\theta).\]
The projection formula now allows us to rewrite the terms of $<S;\gamma_1,\ldots,\gamma_n>_\beta$:
\begin{multline}
	\label{pulledbackterm}
	\int_{[\Mbstack_{g,n}(X,\beta)]^\text{vir}}{\chi_\Gamma}_*F_\Gamma^*(\theta)\cup \ev^*(\gamma_1\otimes\cdots\otimes\gamma_n)=\\
	\int F_\Gamma^*(\theta)\cup\chi_\Gamma^*\left(\ev^*(\gamma_1\otimes\cdots\otimes\gamma_n)\cap[\Mbstack_{g,n}(X,\beta)]^\mathrm{vir}\right).
\end{multline}
We can then calculate $F_\Gamma^*(\theta)$ using
\begin{lemma}[See \RN{6}.3.6 in \cite{maninbookfrobenius}]
 \label{psipullback}
 Let $D_i\in H^2(\Mbstack_{g,n}(X,\beta))$ be defined by taking the closure of the locus of maps whose source curve has two irreducible components such that one of these components is rational, has only the $i$-th marked point, and the restriction of the map to this component is constant. Then
 \[F^*(\psi_i)=\tilde\psi_i-D_i.\]
\end{lemma}
\begin{remark}
This is sufficient for us, since we can express all tautological relations without using $\kappa$-classes when the genus is at most $2$.
In general a class $\kappa_i$ can be replaced by an extra leg with a $\psi^{i+1}$ class on it.
The proof for this is similar to the proof of the Dilaton Equation (see Lemma \ref{dilaton}) which corresponds to the special case $i=0$.
\end{remark}
The \emph{Splitting Lemma} \ref{splittinglemma} gives us a way to express the remaining part of (\ref{pulledbackterm}),
\[\chi_\Gamma^*\left(\ev^*(\gamma_1\otimes\cdots\otimes\gamma_n)\cap[\Mbstack_{g,n}(X,\beta)]^\text{vir}\right),\]
as a sum of products of Gromov-Witten invariants in $F^*(\Mbstack_\Gamma)$.

The morphism $\chi_\Gamma$ is the morphism that glues moduli spaces of stable maps.
Similar to the situation for stable curves, these gluing morphisms can be constructed from the base cases
\begin{align}
	\label{gluemapmap}
	\begin{split}
	\tilde q:\Mbstack_{g_1,n_1+1}(X,\beta_1)\times_X\Mbstack_{g_2,n_2+1}(X,\beta_2)&\rightarrow\Mbstack_{g_1+g_2,n_1+n_2}(X,\beta_1+\beta_2),\\
	\tilde\delta: \Mbstack_{g,n+2}^{\,\delta}(X,\beta)&\rightarrow\Mbstack_{g+1,n}(X,\beta).
\end{split}
\end{align}
We can only glue maps when the points being glued map to the same point in $X$,
i.e. we have fiber squares
\begin{equation*}
\begin{tikzcd}
 \Mbstack_{g_1,n_1+1}(X,\beta_1)\times_X\Mbstack_{g_2,n_2+1}(X,\beta_2) \arrow[r, hook] \arrow[d] & \Mbstack_{g_1,n_1+1}(X,\beta_1)\times\Mbstack_{g_2,n_2+1}(X,\beta_2) \arrow[d,"\ev_{n_1+1} \times \ev_{n_1+2}"] \\
X \arrow[r,hook, "\Delta_X"] & X\times X,
\end{tikzcd}
\end{equation*}
and
\begin{equation*}
\begin{tikzcd}
 \Mbstack_{g,n+2}^{\,\delta}(X,\beta) \arrow[r, hook] \arrow[d] & \Mbstack_{g,n+2}(X,\beta) \arrow[d,"\ev_{n+1} \times \ev_{n+2}"] \\
X \arrow[r,hook, "\Delta_X"] & X\times X.
\end{tikzcd}
\end{equation*}

We can now pull back the formula for the decomposition of the diagonal to obtain an expression for the pullback along $\tilde q$ and $\tilde \delta$.
Coupled with Theorem~13 in \cite{getzlerg2} this gives the Splitting Lemma.

Let \gls{tis} be a homogeneous basis of $H^*(X)$.
The intersection numbers $g_{ef}:=\int_X T_e\cup T_f$ form a matrix $(g_{ef})$.
We write $($\gls{gef}$)$ for the inverse matrix.

\begin{lemma}[Splitting Lemma]
	\label{splittinglemma}
	Let $\tilde q$ be the gluing map (\ref{gluemapmap}) for some given $\beta_1+\beta_2=\beta\in H_2(X)$, we have
\begin{multline*}
	\tilde q^*\left(\ev^*(\gamma_1\otimes\cdots\otimes\gamma_{n_1+n_2})\cap [\Mbstack_{g_1+g_2,n_1+n_2}(X,\beta)]^\mathrm{vir}\right)=\\
\sum_{e,f} g^{ef}
\left(\ev^*(\gamma_1\otimes\cdots\otimes\gamma_{n_1}\otimes T_e)\cap [\Mbstack_{g_1,n_1+1}(X,\beta_1)]^\mathrm{vir}\right)
\otimes \\
\left(\ev^*(T_f\otimes\gamma_{n_1+1}\otimes\cdots\gamma_{n_1+n_2})\cap [\Mbstack_{g_2,n_2+1}(X,\beta_2)]^\mathrm{vir}\right)
,\end{multline*}
and for the self-gluing $\tilde\delta$, we have
\begin{multline*}
	\tilde\delta^*\left(\ev^*(\gamma_1\otimes\cdots\otimes\gamma_n)\cap [\Mbstack_{g+1,n}(X,\beta)]^\mathrm{vir}\right)=\\
\sum_{e,f} g^{ef}
\ev^*(\gamma_1\otimes\cdots\otimes\gamma_n\otimes T_e\otimes T_f)\cap [\Mbstack_{g,n+2}(X,\beta)]^\mathrm{vir}
.\end{multline*}
\end{lemma}

\begin{example}
	When we apply the Splitting Lemma to
	\[<	\left[	\begin{tikzpicture}[baseline,el/.style = {inner sep=2pt, align=left, sloped},every child node/.style={inner sep=1,font=\tiny}]
      \tikzstyle{level 1}=[counterclockwise from=315,level distance=9mm,sibling angle=90]
			\node (A0) [draw,circle,inner sep=1] {$\scriptstyle{0}$} child {node {3}} child {node {1}};
      \tikzstyle{level 1}=[counterclockwise from=180,level distance=9mm,sibling angle=0]
			\node (A1) [draw,circle,inner sep=1,left of=A1] {$\scriptstyle{3}$} child {node {2}};

			\path (A0) edge [bend left=20] (A1);
			\path (A0) edge [bend right=20] (A1);
	\end{tikzpicture}\right];\gamma_1,\gamma_2,\gamma_3>_\beta.\]
	after reordering the points and by contracting the upper edge first, we obtain
	\[(-1)^{|\gamma_2||\gamma_3|}\sum_{\beta_1+\beta_2=\beta}\sum_{e,f}\sum_{e',f'}<\gamma_1\gamma_3 T_{e'}T_e>_{0,\beta_1}	g^{ef}g^{e'f'}<T_fT_{f'}\gamma_2>_{3,\beta_2}.\]
	Note that the result does not depend on the choice of reordering or the choice of which edge to contract first.
\end{example}

When the degree $\beta$ is zero on a rational tail with three special points, the Splitting Lemma has a simple description:
\begin{lemma}
	\label{m03}
	Let $\beta\in H_2(X)$ and $\gamma_1,\ldots,\gamma_n\in H^*(X)$, we have
\begin{multline*}
	\sum_{e,f}\left<\gamma_1 \gamma_2 T_e\right>_{0,0}g^{ef}\left<\tau_{a_0}(T_f)\tau_{a_3}(\gamma_3)\cdots\tau_{a_n}(\gamma_n)\right>_{g,\beta}=\\
\left<\tau_{a_0}(\gamma_1\cup\gamma_2)\tau_{a_3}(\gamma_3)\cdots\tau_{a_n}(\gamma_n)\right>_{g,\beta}.
\end{multline*}
\end{lemma}
\begin{proof}
Since $\Mbstack_{0,3}(X,0)=X$ we have
\[\left<\gamma_1\gamma_2 T_e\right>_{0,0}^X=\int_X \gamma_1\cup\gamma_2\cup T_e=\int_X\sum_i c_i T_i\cup T_e,\]
were we wrote $\gamma_1\cup\gamma_2$ in terms of the generating basis $T_0,\ldots,T_k$.
Now we obtain
\[
\sum_{i,e,f} \left(\int_X c_i T_i\cup T_e\right) g^{ef} T_f=
\sum_{i,e,f} c_i \cdot g_{ie} \cdot g^{ef} T_f=
\sum_{i,f}c_i\cdot\delta_{if} T_f=
	\gamma_1\cup\gamma_2.
\]
\end{proof}

\section{Reconstruction using symbols}
\label{section:symbol}
We develop the notion of a symbol of a tautological relation $L \in R_{g,n}$.
This is done by taking $<L;\gamma_1,\ldots,\gamma_n>_\beta$ and forgetting those terms that can be considered of a lower order in some specific sense.
In order to make this precise we give a new definition for formal Gromov-Witten invariants.

\begin{definition}
	An $n$-pointed \emph{formal Gromov-Witten invariant} of $X$,
\[\left<\tau_{k_1}(\gamma_1)\cdots\tau_{k_n}(\gamma_n)\right>_{g,\beta}^X,\]
	consists of the data $k_1,\ldots,k_n,g\in\mathbb{Z}_{\geq0}$, $\beta\in H_2(X)$, and $\gamma_1,\ldots,\gamma_n\in~H^*(X)$.
	Here the data
	\[\gamma_1,\ldots,\lambda\gamma_i,\gamma_{i+1},\ldots,\gamma_n\]
	is equivalent to
\[\gamma_1,\ldots,\gamma_i,(-1)^{|\gamma_i||\gamma_{i+1}|}\lambda\gamma_{i+1},\ldots,\gamma_n\]
	for any $\lambda\in\mathbb{Q}\setminus\{0\}$ and $1\leq i< n$.
	Denote the set of $n$-pointed formal Gromov-Witten invariants of $X$ by $\mathrm{GW}_n(X)$.
	We write $\mathrm{GW}(X):=\bigcup_n\mathrm{GW}_n(X)$.
\end{definition}
	\begin{definition}
		Let $\mathrm{G}_{g,\beta}^X\subset\mathbb{Q}[\mathrm{GW}(X)]$ be the $\mathbb{Q}$-vector space of linear combinations of formal degree~$\beta$ genus~g Gromov-Witten invariants of $X$.

		We have a realization map
		\[\text{\gls{eta}}:\mathbb{Q}[\mathrm{GW}(X)]\rightarrow\mathbb{Q}\]
		that sends a formal Gromov-Witten invariant to the value of its corresponding Gromov-Witten invariant.

		We now define a new grading on $\mathrm{G}_{g,\beta}^X$ where the degree of an invariant $\left<\tau_{k_1}(\gamma_1)\cdots\tau_{k_n}(\gamma_n)\right>_{g,\beta}^X$ is given by $n-\sum_i k_i$.
\end{definition}
Note that this new grading is different from the usual cohomological grading $\sum_i |\gamma_i|+k_i$.

\begin{definition}
	For $S\in S_{g,n}^r$, we define \gls{Sprim}, the \emph{primitive part} of $S$, by restricting $S$ to those terms that are strata classes for which we have
	\[n(v)-|\theta|=n-r,\]
	where $v$ is a vertex of genus $g$ and $\theta$ is the decoration at $v$.
\end{definition}
\begin{remark}
The \emph{primitive part} of $S$ consists exactly of those terms that are decorated strata classes with rational tail where every vertex besides a single genus $g$ vertex has exactly 3 special points and no decoration.
\end{remark}

The map $<.\, ;\ldots >_\beta$ from Definition \ref{gwpbdefdef} factors through a map $<.\, ;\ldots >_\beta^\mathrm{form}$
\[
	\begin{tikzcd}[column sep=large]
		S_{g,n}^\mathrm{r}\otimes \left(H^*(X)\right)^{\otimes n}\arrow[bend left =25]{rr}{<.\, ;\ldots >_\beta}\arrow{r}{<.\, ;\ldots >_\beta^\mathrm{form}} & \mathbb{Q}[\mathrm{GW}(X)]\arrow{r}{\eta} &  \mathbb{Q}.
 \end{tikzcd}\]
\begin{definition}
	\label{symboldef}
	We fix the degree $\beta\in H_2(X)$ and we define a map
	\[\Sigma:S_{g,n}^r\otimes \left(H^*(X)\right)^{\otimes n}\rightarrow \mathrm{G}_{g,\beta}^X\]
	by taking \gls{sigma}, the \emph{symbol} of $S\in S_{g,n}^r$, to be the restriction of ${<S_\mathrm{prim};\gamma_1,\ldots,\gamma_n>_\beta^\mathrm{form}}$ to those terms for which the degree $\beta$ is concentrated in one genus $g$ Gromov-Witten invariant.
\end{definition}
By Lemma \ref{m03}, the map to $\mathrm{G}_{g,\beta}^X$ is well defined.
	The image of $S_{g,n}^r$ is homogeneous of degree $n-r$.
\begin{notation}
	To save space we will not write the cup products for Gromov-Witten invariants in $\mathrm{G}_{g,\beta}^X$ (e.g. $<\gamma_1\gamma_2>=<\gamma_1\cup\gamma_2>$).
	To prevent the resulting ambiguity we also change the notation for invariants in $\mathrm{G}_{g,\beta}^X$ from \[<\tau_{a_1}(\gamma_1)\cdots\tau_{a_n}(\gamma_n)>_{g,\beta}\]
	to
	\[\text{\gls{symbolgwi}}.\]
	When the genus $g$ and degree $\beta$ are fixed we will leave them out.
\end{notation}

\begin{example}
	\label{symbolexample}
	Consider the linear combination of decorated strata classes in $S_{2,3}^2$
\[
	+3\left[	\begin{tikzpicture}[baseline,el/.style = {inner sep=2pt, align=left, sloped},every child node/.style={inner sep=1,font=\tiny}]
      \tikzstyle{level 1}=[counterclockwise from=315,level distance=9mm,sibling angle=90]
			\node (A0) [draw,circle,inner sep=1] at (0:1.3) {$\scriptstyle{0}$} child {node {1}} child {node {3}};
      \tikzstyle{level 1}=[counterclockwise from=180,level distance=9mm,sibling angle=0]
			\node (A1) [draw,circle,inner sep=1] at (0:0) {$\scriptstyle{2}$} child {node {2}};

      \path (A0) edge [bend left=0.000000] node[el,below,font=\tiny,pos=.8] {$\psi$} (A1);
	\end{tikzpicture}\right]
	+\left[	\begin{tikzpicture}[baseline,el/.style = {inner sep=2pt, align=left, sloped},every child node/.style={inner sep=1,font=\tiny}]
      \tikzstyle{level 1}=[clockwise from=45,level distance=9mm,sibling angle=45]
			\node (A0) [draw,circle,inner sep=1] at (0:1.3) {$\scriptstyle{0}$} child {node {1}} child {node {2}} child {node {3}};
      \tikzstyle{level 1}=[counterclockwise from=180,level distance=9mm,sibling angle=0]
			\node (A1) [draw,circle,inner sep=1] at (0:0) {$\scriptstyle{2}$};

      \path (A0) edge [bend left=0.000000] node[el,below,font=\tiny,pos=.8] {$\psi$} (A1);
	\end{tikzpicture}\right]
	-\left[	\begin{tikzpicture}[baseline,el/.style = {inner sep=2pt, align=left, sloped},every child node/.style={inner sep=1,font=\tiny}]
      \tikzstyle{level 1}=[counterclockwise from=315,level distance=9mm,sibling angle=90]
			\node (A0) [draw,circle,inner sep=1] at (0:1.3) {$\scriptstyle{0}$} child {node {2}} child {node {3}};
      \tikzstyle{level 1}=[counterclockwise from=270,level distance=9mm,sibling angle=120]
      \node (A1) [draw,circle,inner sep=1] at (0:0) {$\scriptstyle{0}$} child {node {1}};
      \tikzstyle{level 1}=[counterclockwise from=240,level distance=9mm,sibling angle=0]
			\node (A2) [draw,circle,inner sep=1] at (0:-1.3) {$\scriptstyle{2}$};

      \path (A0) edge [bend left=0.000000] (A1);
      \path (A1) edge [bend left=0.000000] (A2);
	\end{tikzpicture}\right]
	-2\left[	\begin{tikzpicture}[baseline,el/.style = {inner sep=2pt, align=left, sloped},every child node/.style={inner sep=1,font=\tiny}]
      \tikzstyle{level 1}=[counterclockwise from=315,level distance=9mm,sibling angle=90]
			\node (A0) [draw,circle,inner sep=1] at (0:1.3) {$\scriptstyle{1}$} child {node {1}} child {node {3}};
      \tikzstyle{level 1}=[counterclockwise from=270,level distance=9mm,sibling angle=120]
      \node (A1) [draw,circle,inner sep=1] at (0:0) {$\scriptstyle{0}$} child {node {2}};
      \tikzstyle{level 1}=[counterclockwise from=240,level distance=9mm,sibling angle=0]
			\node (A2) [draw,circle,inner sep=1] at (0:-1.3) {$\scriptstyle{1}$};

      \path (A0) edge [bend left=0.000000] (A1);
      \path (A1) edge [bend left=0.000000] (A2);
	\end{tikzpicture}\right]
\]
	The primitive part consists of the first and third term. Its symbol is
	\[	(-1)^{|\gamma_2||\gamma_3|}\cdot3\left<\psi\gamma_1\gamma_3,\gamma_2\right>-\left<\gamma_1 \gamma_2 \gamma_3\right>.\]
\end{example}

\begin{definition}
	We introduce an equivalence relation on $\mathrm{G}_{g,\beta}^X$ as follows.
	Let $x\in \mathrm{G}_{g,\beta}^X$ be homogeneous of degree $d$.
	We say $x\text{\gls{equiv}} 0$ if there is a $P\in\mathbb{Q}[\mathrm{GW}(X)]$ such that $\eta(x-P)=0$ and $P$ is a polynomial in formal Gromov-Witten invariants
\[<\tau_{a_1}(\gamma_1)\cdots\tau_{a_m}(\gamma_m)>_{g',\beta'}\]
for which $g'\leq g$, $\beta'\leq \beta$, or $m-\sum_i a_i\leq d$, and at least one of these three inequalities is not an equality.

(For $\beta,\beta'\in H_2(X)$ we say that $\beta'<\beta$ if there exists a $\beta''\in H_2(X)$ such that $\beta''$ is the image under the cycle map of a nonzero effective $1$-cycle and $\beta'+\beta''=\beta$.)
\end{definition}
By construction, if $R\in R_{g,n}^r$, then $\Sigma(R)(\gamma_1,\ldots,\gamma_n)\sim0$.

We can now state a combined version of the String and Divisor Equations (Proposition 12 in \cite{getzlerg2}) for symbols:
\begin{lemma}[]
	\label{genstring}
	Let $a\in H^*(X)$ with $|a|\leq 2$ then
	\[\left<\psi^{k_1}\gamma_1,\ldots\psi^{k_n}\gamma_n,a\right>_{g,\beta}\sim	\sum_{i=1}^n \left<\psi^{k_1}\gamma_1,\ldots,\psi^{k_i-1}\gamma_i a,\ldots,\psi^{k_n}\gamma_n\right>_{g,\beta},\]
{\small(Here it is understood that the terms of the sum where $k_i-1$ becomes negative are zero.)}
	except when $\beta=0$ and $(g,n)\in\{(0,2),(1,0)\}$, in which case the only (possibly) nonzero invariants are
\begin{align*}
	\left<\gamma_1\gamma_2\gamma_3\right>_{0,0}^X&=\int_X \gamma_1\cup\gamma_2\cup\gamma_3, \\
	\left<\gamma\right>_{1,0}^X&=-\frac{1}{24}\int_X c_{\mathrm{dim}(X)-1}(X)\cup\gamma.
\end{align*}
\end{lemma}
Similarly the Dilaton Equation (\RN{6}.5.3 in \cite{maninbookfrobenius}) becomes:
\begin{lemma}
	\label{dilaton}
	We have
	\[\left<\psi^{a_1}\gamma_1,\ldots,\psi^{a_n}\gamma_n,\psi\right>_{g,\beta}\sim(2g-2+n)\left<\psi^{a_1}\gamma_1,\ldots,\psi^{a_n}\gamma_n\right>_{g,\beta},\]
	except when $\beta=0$, $g=1$ and $n=0$, in which case we have
	\[\left<\psi\gamma\right>_{1,0}^X=-\frac{1}{24}\int_X c_{\mathrm{dim}(X)}(X)\cup\gamma.\]
\end{lemma}

\begin{definition}
	\label{pbsymbdef}
	Let $S\in S_{g,n}^r$, by abuse of notation we define the \emph{pullback of $\Sigma(S)(\gamma_1,\ldots,\gamma_n)$ with cohomology class $\gamma_{n+1}$} to be
	\[\Sigma(\pi^*(S))(\gamma_1,\ldots,\gamma_n,\gamma_{n+1}).\]
\end{definition}

By the String Equation (see 2.59 in \cite{harrismorrison}), the pullback of $\left<\psi^{a_1}\gamma_1,\ldots,\psi^{a_n}\gamma_n\right>$ with class $\psi_{n+1}$ equals
\[\left<\psi^{a_1}\gamma_1,\ldots,\psi^{a_n}\gamma_n,\gamma_{n+1}\right>-\sum_{i=1}^n \left<\psi^{a_1}\gamma_1,\ldots,\psi^{a_i-1}\gamma_{i}\gamma_{n+1},\ldots,\psi^{a_n}\gamma_n\right>.\]

For the purpose of this paper we say a Gromov-Witten invariant can be \emph{reconstructed} from a set of Gromov-Witten invariants $S$ if it can be calculated recursively by taking the set $S$ as initial values and applying the following relations:
	\begin{itemize}
		\item $<R;\gamma_1,\ldots,\gamma_n>_\beta$ for any $R\in R_{g,n}$, $\beta\in H_2(X)$, and $\gamma_1,\ldots,\gamma_n\in H^*(X)$,
		\item The string, divisor, and dilaton equations and their special cases (See Lemmas~\ref{genstring}, \ref{dilaton}).
	\end{itemize}
	Our way to prove that $GW(X)$ can be reconstructed from $S$ is to show that $w\sim 0$ for all $w\in\mathrm{GW}(X)\setminus S$.

\section{Reconstruction for genus~$0$}
We will write down the proof of the first reconstruction theorem in \cite{kontmaninrecon} using the language of symbols.
Unlike the original statement we will also allow for odd cohomology.

\begin{theorem}
	\label{g0reconstruction}
	If $H^*(X)$ is generated as a ring by $H^{\leq2}(X)$ then all genus zero Gromov-Witten invariants can be reconstructed recursively from primitive invariants with at most 2 points.
\end{theorem}
\begin{proof}
	By Proposition \ref{psivanishing} we only need to consider primitive invariants.
Let us have a 3 pointed Gromov-Witten invariant $<a,b,c>$, with $a,b,c\in H^{*}(X)$.
We claim that
\begin{equation}
	\label{g0re1}
0\sim<a,b,c>.
\end{equation}
We construct symbols from relations and by using Lemmas \ref{dilaton} and \ref{genstring}.
The equations in these Lemmas still hold when pulled back, so we can pull back (\ref{g0re1}) and insert cohomology classes $\gamma_4,\ldots,\gamma_n$ to obtain
\[0\sim<a,b,c,\gamma_4,\ldots,\gamma_n>,\]
which proves the theorem.

To prove the claim we use induction on $|c|$.
If $|c|\leq 2$ then we apply Lemma~\ref{genstring}.
If $|c|>2$, then using the hypothesis on $X$, we can write $c=c'd$ with $0<|d|\leq 2$.
We now take the tautological relation $L_{0,4}\in R^1_{0,4}$ known as the WDVV relation:
\[\left[\begin{tikzpicture}[baseline,el/.style = {inner sep=2pt, align=left, sloped},every child node/.style={inner sep=1,font=\tiny}]
      \tikzstyle{level 1}=[counterclockwise from=315,level distance=9mm,sibling angle=90]
			\node (A0) [draw,circle,inner sep=1] at (0:1.3) {$\scriptstyle{0}$} child {node {4}} child {node {3}};
      \tikzstyle{level 1}=[counterclockwise from=135,level distance=9mm,sibling angle=90]
			\node (A1) [draw,circle,inner sep=1] at (0:0) {$\scriptstyle{0}$} child {node {2}} child {node {1}};

      \path (A0) edge (A1);
\end{tikzpicture}\right]
-\left[\begin{tikzpicture}[baseline,el/.style = {inner sep=2pt, align=left, sloped},every child node/.style={inner sep=1,font=\tiny}]
      \tikzstyle{level 1}=[counterclockwise from=315,level distance=9mm,sibling angle=90]
			\node (A0) [draw,circle,inner sep=1] at (0:1.3) {$\scriptstyle{0}$} child {node {4}} child {node {2}};
      \tikzstyle{level 1}=[counterclockwise from=135,level distance=9mm,sibling angle=90]
			\node (A1) [draw,circle,inner sep=1] at (0:0) {$\scriptstyle{0}$} child {node {3}} child {node {1}};

      \path (A0) edge (A1);
	\end{tikzpicture}\right]
=0
\]
We take its symbol $\Sigma(S)(a,b,c',d)$ to obtain
\[0\sim<a,b,c'd>+<ab,c',d>-(-1)^{|b||c'|}(<a,c',bd>+<ac',b,d>.\]
Using Lemma~\ref{genstring} this becomes
\[0\sim<a,b,c'd>-(-1)^{|b||c'|}<a,c',bd>.\]
Since $|c'|<|c|$ this ends our induction step.
\end{proof}
\section{Reconstruction for genus~$1$}
In \cite{getzlerg1} Getzler found a new tautological relation in $R^2(\Mbstack_{1,4})$ and used it to prove the following reconstruction theorem for genus~$1$.
There is a step missing in the proof, namely in Corollary~3.3 of \cite{getzlerg1} the initial conditions for the difference equation are missing.
So in this section we will give our own proof of the theorem.
This proof uses the same type of methods as are used by Getzler.

We fix an ample divisor $\omega\in H^2(X)$ and define the primitive cohomology of $X$ to be
\[P^i(X):=\mathrm{coker}(H^{i-2}(X,\mathbb{C})\xrightarrow{\cdot\cup\omega}H^{i}(X,\mathbb{C})).\]

\begin{theorem}
	\label{g1reconstruction}
	If $P^i(X)=0$ for $i>2$, then all genus one Gromov-Witten invariants can be reconstructed recursively from primary genus~$1$ invariants with at most 1 point and all genus zero invariants.
\end{theorem}

We will use the following very basic result about difference equations.
\begin{lemma}
	\label{diffeq}
	Let $h\in G_{g,\beta}^X$ and let $I$ be the set $\{0,1,\ldots,N\}$.
	Let $f:I\rightarrow G_{g,\beta}^X$ be a solution to the equation
	\[f(i+2)-2 f(i+1)+f(i)-h = 0,\]
	for all $0\leq i\leq N-2$.
	If $l\in I$, then $f$ satisfies the formula
	\[l f(i) =i  f(l) +(l-i) f(0) - \frac{li(l-i)}{2} h.\]
\end{lemma}
\begin{proof}
	Subtract a shifted version of the equation to obtain a homogeneous 3rd order difference equation, which can be solved.
\end{proof}

Now we can prove the reconstruction theorem:
\begin{proof}[Proof of Theorem~\ref{g1reconstruction}]
By Proposition \ref{psivanishing} we only need to consider primitive invariants.
The primitive part of Getzler's relation $L_{1,4}\in R^2_{1,4}$ is
\[3\left[\begin{tikzpicture}[baseline,el/.style = {inner sep=2pt, align=left, sloped},every child node/.style={inner sep=1,font=\tiny}]
      \tikzstyle{level 1}=[counterclockwise from=315,level distance=9mm,sibling angle=90]
			\node (A0) [draw,circle,inner sep=1] at (0:1.3) {$\scriptstyle{0}$} child {node {}} child {node {}};
      \tikzstyle{level 1}=[counterclockwise from=135,level distance=9mm,sibling angle=90]
			\node (A1) [draw,circle,inner sep=1] at (0:0) {$\scriptstyle{1}$};
      \tikzstyle{level 1}=[counterclockwise from=135,level distance=9mm,sibling angle=90]
			\node (A2) [draw,circle,inner sep=1] at (0:-1.3) {$\scriptstyle{0}$} child {node {}} child {node {}};

      \path (A0) edge (A1);
      \path (A1) edge (A2);
\end{tikzpicture}\right]
-4\left[\begin{tikzpicture}[baseline,el/.style = {inner sep=2pt, align=left, sloped},every child node/.style={inner sep=1,font=\tiny}]
      \tikzstyle{level 1}=[counterclockwise from=0,level distance=9mm,sibling angle=90]
			\node (A0) [draw,circle,inner sep=1] at (0:1.3) {$\scriptstyle{1}$} child {node {}};
      \tikzstyle{level 1}=[counterclockwise from=90,level distance=9mm,sibling angle=90]
			\node (A1) [draw,circle,inner sep=1] at (0:0) {$\scriptstyle{0}$} child {node {}};
      \tikzstyle{level 1}=[counterclockwise from=135,level distance=9mm,sibling angle=90]
			\node (A2) [draw,circle,inner sep=1] at (0:-1.3) {$\scriptstyle{0}$} child {node {}} child {node {}};

      \path (A0) edge (A1);
      \path (A1) edge (A2);
\end{tikzpicture}\right]
.\]
Let $a,b\in H^{\leq2}(X)$.
Fix an integer $l\geq2$, we define
\[f(i) :=\left<a \omega^{i},b \omega^{l-i}\right>,\qquad
h :=\left<a b\omega^{l-2},\omega^2\right>.\]
Applying Lemma~\ref{genstring} to the symbol of $L_{1,4}$ gives
\[\Sigma(L_{1,4})(a\omega^k,b\omega^{l-2-k},\omega,\omega)\sim
	f(k+2)-2 f(k+1)+f(k)-h,\]
	for every $0\leq k\leq l-2$.
	Again by Lemma~\ref{genstring} $f(0)\sim f(l)\sim 0$ so by Lemma~\ref{diffeq} we have
	\begin{equation}
		\label{g1w2}
	<a\omega^i,b\omega^j>\sim \frac{ij}{2}<ab\omega^{i+j-2},\omega^2>.
	\end{equation}
	We substitute $b\rightarrow\omega, j\rightarrow 1$ or $b\rightarrow 1,j\rightarrow 2$ to obtain
	\[<a\omega^i,\omega^2>\sim \frac{i}{2}<a\omega^{i},\omega^2>\sim i<a\omega^{i},\omega^2>,\]
	thus
	\begin{equation}
		\label{g1w2=0}
	<a\omega^i,\omega^2>\sim 0.
\end{equation}
Using our hypothesis on $X$ we can write $ab\omega^{i+j-2}=c\omega^k$ for some $k\in\mathbb{Z}_{\geq0}$ and $c\in H^{\leq2}(X)$.
	So we can apply (\ref{g1w2=0}) to the right-hand side of (\ref{g1w2}) to obtain
\begin{equation}
	\label{g1result}
		<a\omega^i,b\omega^j>\sim 0.
\end{equation}
We can repeatedly pull back (\ref{g1result}) and insert cohomology classes $\gamma_3,\ldots,\gamma_n$ to obtain
\[<a\omega^i,b\omega^j,\gamma_3,\ldots,\gamma_n>\sim 0.\]
\end{proof}

\section{Reconstruction for genus~$2$}
\label{genus2section}
In this chapter we prove our reconstruction theorem for genus~$2$ Gromov-Witten invariants.

\begin{theorem}
	\label{g2reconstruction}
	If $P^i(X)=0$ for $i>2$, then all (including descendant) genus two Gromov-Witten invariants can be reconstructed recursively from genus two invariants with at most two points and invariants of lower genus.
\end{theorem}
\begin{proof}
	By Proposition \ref{psivanishing} we only need to consider primitive invariants and invariants with a single $\psi$-class.
	Let $a,b,c\in H^*(X)$ such that $|a|,|b|,|c|\leq2$ and let $i,j,k\geq0$, then
	\begin{equation*}
	\left<a \omega^i,b \omega^j,c\omega^k\right>\sim0,
	\end{equation*}
	and
	\[\left<\psi a\omega^i,b\omega^j,c\omega^k\right>\sim \text{a linear combination of 2 pointed primary invariants}.\]
	This is proven in Lemmas \ref{primitivereconstruction} and \ref{nonprimitivereconstruction}.

	By pulling back the two relations above we can always express an invariant with 3 or more points as a linear combination of invariants with lower degree, genus or number of points.
\end{proof}

We will now prove the Lemmas required for Theorem~\ref{g2reconstruction}.
We frequently apply Lemma~\ref{genstring} throughtout the proof without explicitly mentioning it every time.
The proof requires some buildup where for simplicity we first work with Gromov-Witten invariants that have only one or two points.

Our first aim is to express descendant invariants in terms of primary invariants.
\begin{lemma}
	\label{243lemma}
	Let $a,b\in H^{\leq2}(X)$, the following holds for all $i,j\geq0$
	\begin{align*}
	(i+j)\left<\psi a\omega^{i},b\omega^{j}\right>
	\sim &
	i\left<a b \omega^{i+j}\right>
	+j\left<\psi a,b \omega^{i+j}\right>
	\\ &
	-\frac{ij(i+j)}{2}\Big(
	\left<\psi\omega^{2},a b\omega^{i+j-2}\right>
	\\ &
	-2\left<\psi\omega,a b\omega^{i+j-1}\right>
	+3\left<a b \omega^{i+j}\right>
	\Big).
	\end{align*}
\end{lemma}
\begin{proof}
	Let us fix the number $l:=i+j$.
	For $i=0$ or $j=0$ the statement follows from Lemma~\ref{genstring} so we can assume $l\geq2$.
	Using our computer program, we find a relation $L_{2,4}\in R^3_{2,4}$ with primitive part
\[
\left[	\begin{tikzpicture}[baseline,el/.style = {inner sep=2pt, align=left, sloped},every child node/.style={inner sep=1,font=\tiny}]
      \path(0,0) ellipse (2 and 2);
      \tikzstyle{level 1}=[counterclockwise from=0,level distance=9mm,sibling angle=0]
      \node (A0) [draw,circle,inner sep=1] at (0:1) {$\scriptstyle{0}$} child {node {}};
      \tikzstyle{level 1}=[counterclockwise from=60,level distance=9mm,sibling angle=120]
      \node (A1) [draw,circle,inner sep=1] at (120:1) {$\scriptstyle{0}$} child {node {}} child {node {}};
      \tikzstyle{level 1}=[counterclockwise from=240,level distance=9mm,sibling angle=0]
      \node (A2) [draw,circle,inner sep=1] at (240:1) {$\scriptstyle{2}$} child {node {2}};

      \path (A0) edge [bend left=0.000000] (A1);
      \path (A0) edge [bend left=0.000000] node[el,above,font=\tiny,pos=.8] {$\psi$} (A2);
    \end{tikzpicture}\right]
+3\left[	\begin{tikzpicture}[baseline,el/.style = {inner sep=2pt, align=left, sloped},every child node/.style={inner sep=1,font=\tiny}]
      \path(0,0) ellipse (2 and 2);
      \tikzstyle{level 1}=[counterclockwise from=0,level distance=9mm,sibling angle=0]
			\node (A0) [draw,circle,inner sep=1] at (0:1) {$\scriptstyle{0}$} child {node {}};
      \tikzstyle{level 1}=[counterclockwise from=60,level distance=9mm,sibling angle=120]
      \node (A1) [draw,circle,inner sep=1] at (120:1) {$\scriptstyle{0}$} child {node {}} child {node {2}};
      \tikzstyle{level 1}=[counterclockwise from=240,level distance=9mm,sibling angle=0]
			\node (A2) [draw,circle,inner sep=1] at (240:1) {$\scriptstyle{2}$} child {node {$\psi$}};

      \path (A0) edge [bend left=0.000000] (A1);
      \path (A0) edge [bend left=0.000000] (A2);
	\end{tikzpicture}\right]
-3\left[	\begin{tikzpicture}[baseline,el/.style = {inner sep=2pt, align=left, sloped},every child node/.style={inner sep=1,font=\tiny}]
      \path(0,0) ellipse (2 and 2);
      \tikzstyle{level 1}=[counterclockwise from=-60,level distance=9mm,sibling angle=120]
      \node (A0) [draw,circle,inner sep=1] at (0:1) {$\scriptstyle{0}$} child {node {}} child {node {}};
      \tikzstyle{level 1}=[counterclockwise from=60,level distance=9mm,sibling angle=120]
      \node (A1) [draw,circle,inner sep=1] at (120:1) {$\scriptstyle{0}$} child {node {}} child {node {2}};
      \node (A2) [draw,circle,inner sep=1] at (240:1) {$\scriptstyle{2}$};
      \path (A0) edge [bend left=0.000000] node[el,below,font=\tiny,pos=.8] {$\psi$} (A2);
      \path (A1) edge [bend left=0.000000] (A2);
	\end{tikzpicture}\right]
-3\left[
	\begin{tikzpicture}[baseline,el/.style = {inner sep=2pt, align=left, sloped},every child node/.style={inner sep=1,font=\tiny}]
      \path(0,0) ellipse (2 and 2);
      \tikzstyle{level 1}=[counterclockwise from=0,level distance=9mm,sibling angle=120]
			\node (A0) [draw,circle,inner sep=1] at (0:1) {$\scriptstyle{0}$} child {node {}};
      \tikzstyle{level 1}=[counterclockwise from=90,level distance=9mm,sibling angle=120]
			\node (A1) [draw,circle,inner sep=1] at (90:1) {$\scriptstyle{0}$} child {node {}};
      \tikzstyle{level 1}=[counterclockwise from=120,level distance=9mm,sibling angle=120]
      \node (A2) [draw,circle,inner sep=1] at (180:1) {$\scriptstyle{0}$} child {node {}} child {node {2}};
      \node (A3) [draw,circle,inner sep=1] at (270:1) {$\scriptstyle{2}$};

			\path (A0) edge [bend left=0.000000] (A1);
      \path (A1) edge [bend left=0.000000] (A2);
      \path (A0) edge [bend left=0.000000] (A3);
\end{tikzpicture}
\right].
\]
	We define
	\[f(i) :=\left<\psi a  \omega^{i},b\omega^{l-i}\right>,	\quad
		h :=\left<\psi\omega^2,a b\omega^{l-2}\right>
	-2\left<\psi\omega,a b\omega^{l-1}\right>
	+3\left<a  b \omega^{l}\right>,\]
	in order to obtain
	\[\Sigma(L_{2,4})(a\omega^k,b\omega^{l-2-k},\omega,\omega)\sim
	f(k+2)-2 f(k+1)+f(k)-h,\]
	for every $0\leq k\leq l-2$.
	Applying Lemma~\ref{diffeq} gives the desired formula.
\end{proof}
In order to get rid of the term with $\psi\omega^2$ in Lemma~\ref{243lemma}, we do the following:
\begin{corollary}
	\label{psiomega2}
	Let $b\in H^{\leq2}(X)$, the following holds for all $j\geq0$
	\begin{equation*}
		(j+1)(j+2)\left<\psi\omega^{2},b\omega^{j}\right>
	\sim
	2j(j+2)\left<\psi\omega,b \omega^{j+1}\right>
	-(3j^2+3j-2)\left<b\omega^{i+2}\right>
	.
	\end{equation*}
\end{corollary}
\begin{proof}
	Substitute $a\rightarrow\omega$ and $i\rightarrow 1$ in Lemma~\ref{243lemma}.
\end{proof}
Now we want to get rid of the term in Lemma~\ref{243lemma} with $\psi\omega$ in it.
However, in order to get rid of terms with a $\psi\omega$ we need to consider 3-pointed invariants.
\begin{lemma}
	\label{psiomega}
	Let $a,b\in H^{\leq2}(X)$, the following holds for all $i,j\geq0$
	\begin{align*}
		(i+j)\left<\psi\omega,a\omega^i,b\omega^{j}\right>
		\sim &
		2(i+j)\left<a\omega^{i+1},b\omega^{j}\right>+2(i+j)\left<a\omega^i,b\omega^{j+1}\right>
		\\ &
		-i\left<a\omega^{i+j},b\omega\right>
		-j\left<a\omega,b\omega^{i+j}\right>
		\\ &
		-\frac{ij(i+j)}{6}\Big(
			3\left<a b\omega^{i+j-1},\omega^{2}\right>-\left<a b\omega^{i+j-2},\omega^{3}\right>
		\big).
	\end{align*}
\end{lemma}
\begin{proof}
	Let us write $l:=i+j$.
	For $i=0$ or $j=0$ it is trivial so we can assume $l\geq2$.
	Using our computer program, we find a relation $L_{2,5}\in R_{2,5}^3$ with primitive part
\[
	2(\left[	\begin{tikzpicture}[baseline,el/.style = {inner sep=2pt, align=left, sloped},every child node/.style={inner sep=1,font=\tiny}]
      \path(0,0) ellipse (2 and 2);
      \tikzstyle{level 1}=[counterclockwise from=-30,level distance=9mm,sibling angle=120]
			\node (A0) [draw,circle,inner sep=1] at (0:1) {$\scriptstyle{0}$} child {node {1}} child {node {}};
      \tikzstyle{level 1}=[counterclockwise from=60,level distance=9mm,sibling angle=120]
      \node (A1) [draw,circle,inner sep=1] at (120:1) {$\scriptstyle{0}$} child {node{}} child {node {2}};
      \tikzstyle{level 1}=[counterclockwise from=240,level distance=9mm,sibling angle=0]
      \node (A2) [draw,circle,inner sep=1] at (240:1) {$\scriptstyle{2}$} child {node {}};

      \path (A1) edge [bend left=0.000000] (A2);
      \path (A0) edge [bend left=0.000000] node[el,below,font=\tiny,pos=.8] {$\psi$} (A2);
    \end{tikzpicture}\right]
-\left[	\begin{tikzpicture}[baseline,el/.style = {inner sep=2pt, align=left, sloped},every child node/.style={inner sep=1,font=\tiny}]
      \path(0,0) ellipse (2 and 2);
      \tikzstyle{level 1}=[counterclockwise from=-30,level distance=9mm,sibling angle=120]
			\node (A0) [draw,circle,inner sep=1] at (0:1) {$\scriptstyle{0}$} child {node {1}} child {node {}};
      \tikzstyle{level 1}=[counterclockwise from=60,level distance=9mm,sibling angle=120]
      \node (A1) [draw,circle,inner sep=1] at (120:1) {$\scriptstyle{0}$} child {node {}} child {node {2}};
      \tikzstyle{level 1}=[counterclockwise from=240,level distance=9mm,sibling angle=0]
      \node (A2) [draw,circle,inner sep=1] at (240:1) {$\scriptstyle{2}$} child {node {$\psi$}};

      \path (A1) edge [bend left=0.000000] (A2);
      \path (A0) edge [bend left=0.000000] (A2);
	\end{tikzpicture}\right])
	-2(\left[	\begin{tikzpicture}[baseline,el/.style = {inner sep=2pt, align=left, sloped},every child node/.style={inner sep=1,font=\tiny}]
      \path(0,0) ellipse (2 and 2);
      \tikzstyle{level 1}=[counterclockwise from=0,level distance=9mm,sibling angle=120]
			\node (A0) [draw,circle,inner sep=1] at (0:1) {$\scriptstyle{0}$} child {node {}};
      \tikzstyle{level 1}=[counterclockwise from=60,level distance=9mm,sibling angle=120]
      \node (A1) [draw,circle,inner sep=1] at (120:1) {$\scriptstyle{0}$} child {node {1}} child {node {2}};
      \tikzstyle{level 1}=[counterclockwise from=270,level distance=9mm,sibling angle=-120]
			\node (A2) [draw,circle,inner sep=1] at (240:1) {$\scriptstyle{2}$} child {node {}} child {node {}};

      \path (A0) edge [bend left=0.000000] (A1);
      \path (A0) edge [bend left=0.000000] node[el,below,font=\tiny,pos=.8] {$\psi$} (A2);
    \end{tikzpicture}\right]
-\left[	\begin{tikzpicture}[baseline,el/.style = {inner sep=2pt, align=left, sloped},every child node/.style={inner sep=1,font=\tiny}]
      \path(0,0) ellipse (2 and 2);
      \tikzstyle{level 1}=[counterclockwise from=0,level distance=9mm,sibling angle=120]
			\node (A0) [draw,circle,inner sep=1] at (0:1) {$\scriptstyle{0}$} child {node {}};
      \tikzstyle{level 1}=[counterclockwise from=60,level distance=9mm,sibling angle=120]
      \node (A1) [draw,circle,inner sep=1] at (120:1) {$\scriptstyle{0}$} child {node {1}} child {node {2}};
      \tikzstyle{level 1}=[counterclockwise from=270,level distance=9mm,sibling angle=-120]
			\node (A2) [draw,circle,inner sep=1] at (240:1) {$\scriptstyle{2}$} child {node {$\psi$}} child {node {}};

      \path (A0) edge [bend left=0.000000] (A1);
      \path (A0) edge [bend left=0.000000] (A2);
	\end{tikzpicture}\right])
\]\[
	-(\left[	\begin{tikzpicture}[baseline,el/.style = {inner sep=2pt, align=left, sloped},every child node/.style={inner sep=1,font=\tiny}]
      \path(0,0) ellipse (2 and 2);
      \tikzstyle{level 1}=[counterclockwise from=0,level distance=9mm,sibling angle=120]
			\node (A0) [draw,circle,inner sep=1] at (0:1) {$\scriptstyle{0}$} child {node {}};
      \tikzstyle{level 1}=[counterclockwise from=60,level distance=9mm,sibling angle=120]
      \node (A1) [draw,circle,inner sep=1] at (120:1) {$\scriptstyle{0}$} child {node {}} child {node {2}};
      \tikzstyle{level 1}=[counterclockwise from=270,level distance=9mm,sibling angle=-120]
			\node (A2) [draw,circle,inner sep=1] at (240:1) {$\scriptstyle{2}$} child {node {}} child {node {1}};

      \path (A0) edge [bend left=0.000000] (A1);
      \path (A0) edge [bend left=0.000000] node[el,below,font=\tiny,pos=.8] {$\psi$} (A2);
    \end{tikzpicture}\right]
-\left[	\begin{tikzpicture}[baseline,el/.style = {inner sep=2pt, align=left, sloped},every child node/.style={inner sep=1,font=\tiny}]
      \path(0,0) ellipse (2 and 2);
      \tikzstyle{level 1}=[counterclockwise from=0,level distance=9mm,sibling angle=120]
			\node (A0) [draw,circle,inner sep=1] at (0:1) {$\scriptstyle{0}$} child {node {}};
      \tikzstyle{level 1}=[counterclockwise from=60,level distance=9mm,sibling angle=120]
      \node (A1) [draw,circle,inner sep=1] at (120:1) {$\scriptstyle{0}$} child {node {}} child {node {2}};
      \tikzstyle{level 1}=[counterclockwise from=270,level distance=9mm,sibling angle=-120]
			\node (A2) [draw,circle,inner sep=1] at (240:1) {$\scriptstyle{2}$} child {node {$\psi$}} child {node {1}};

      \path (A0) edge [bend left=0.000000] (A1);
      \path (A0) edge [bend left=0.000000] (A2);
	\end{tikzpicture}\right])
	-(\left[	\begin{tikzpicture}[baseline,el/.style = {inner sep=2pt, align=left, sloped},every child node/.style={inner sep=1,font=\tiny}]
      \path(0,0) ellipse (2 and 2);
      \tikzstyle{level 1}=[counterclockwise from=0,level distance=9mm,sibling angle=120]
			\node (A0) [draw,circle,inner sep=1] at (0:1) {$\scriptstyle{0}$} child {node {}};
      \tikzstyle{level 1}=[counterclockwise from=60,level distance=9mm,sibling angle=120]
      \node (A1) [draw,circle,inner sep=1] at (120:1) {$\scriptstyle{0}$} child {node {}} child {node {1}};
      \tikzstyle{level 1}=[counterclockwise from=270,level distance=9mm,sibling angle=-120]
			\node (A2) [draw,circle,inner sep=1] at (240:1) {$\scriptstyle{2}$} child {node {}} child {node {2}};

      \path (A0) edge [bend left=0.000000] (A1);
      \path (A0) edge [bend left=0.000000] node[el,below,font=\tiny,pos=.8] {$\psi$} (A2);
    \end{tikzpicture}\right]
-\left[	\begin{tikzpicture}[baseline,el/.style = {inner sep=2pt, align=left, sloped},every child node/.style={inner sep=1,font=\tiny}]
      \path(0,0) ellipse (2 and 2);
      \tikzstyle{level 1}=[counterclockwise from=0,level distance=9mm,sibling angle=120]
			\node (A0) [draw,circle,inner sep=1] at (0:1) {$\scriptstyle{0}$} child {node {}};
      \tikzstyle{level 1}=[counterclockwise from=60,level distance=9mm,sibling angle=120]
      \node (A1) [draw,circle,inner sep=1] at (120:1) {$\scriptstyle{0}$} child {node {}} child {node {1}};
      \tikzstyle{level 1}=[counterclockwise from=270,level distance=9mm,sibling angle=-120]
			\node (A2) [draw,circle,inner sep=1] at (240:1) {$\scriptstyle{2}$} child {node {$\psi$}} child {node {2}};

      \path (A0) edge [bend left=0.000000] (A1);
      \path (A0) edge [bend left=0.000000] (A2);
	\end{tikzpicture}\right])
\]\[
	-(\left[	\begin{tikzpicture}[baseline,el/.style = {inner sep=2pt, align=left, sloped},every child node/.style={inner sep=1,font=\tiny}]
      \path(0,0) ellipse (2 and 2);
      \tikzstyle{level 1}=[counterclockwise from=0,level distance=9mm,sibling angle=120]
			\node (A0) [draw,circle,inner sep=1] at (0:1) {$\scriptstyle{0}$} child {node {}};
      \tikzstyle{level 1}=[counterclockwise from=60,level distance=9mm,sibling angle=120]
      \node (A1) [draw,circle,inner sep=1] at (120:1) {$\scriptstyle{0}$} child {node {}} child {node {1}};
      \tikzstyle{level 1}=[counterclockwise from=270,level distance=9mm,sibling angle=-120]
			\node (A2) [draw,circle,inner sep=1] at (240:1) {$\scriptstyle{2}$} child {node {2}} child {node {}};

      \path (A0) edge [bend left=0.000000] (A1);
      \path (A0) edge [bend left=0.000000] node[el,below,font=\tiny,pos=.8] {$\psi$} (A2);
    \end{tikzpicture}\right]
-\left[	\begin{tikzpicture}[baseline,el/.style = {inner sep=2pt, align=left, sloped},every child node/.style={inner sep=1,font=\tiny}]
      \path(0,0) ellipse (2 and 2);
      \tikzstyle{level 1}=[counterclockwise from=0,level distance=9mm,sibling angle=120]
			\node (A0) [draw,circle,inner sep=1] at (0:1) {$\scriptstyle{0}$} child {node {}};
      \tikzstyle{level 1}=[counterclockwise from=60,level distance=9mm,sibling angle=120]
      \node (A1) [draw,circle,inner sep=1] at (120:1) {$\scriptstyle{0}$} child {node {}} child {node {1}};
      \tikzstyle{level 1}=[counterclockwise from=270,level distance=9mm,sibling angle=-120]
			\node (A2) [draw,circle,inner sep=1] at (240:1) {$\scriptstyle{2}$} child {node {$2\,\psi$}} child {node {}};

      \path (A0) edge [bend left=0.000000] (A1);
      \path (A0) edge [bend left=0.000000] (A2);
	\end{tikzpicture}\right])
-\left[	\begin{tikzpicture}[baseline,el/.style = {inner sep=2pt, align=left, sloped},every child node/.style={inner sep=1,font=\tiny}]
      \path(0,0) ellipse (2 and 2);
      \tikzstyle{level 1}=[counterclockwise from=0,level distance=9mm,sibling angle=120]
			\node (A0) [draw,circle,inner sep=1] at (0:1) {$\scriptstyle{0}$} child {node {}};
      \tikzstyle{level 1}=[counterclockwise from=90,level distance=9mm,sibling angle=120]
			\node (A1) [draw,circle,inner sep=1] at (90:1) {$\scriptstyle{0}$} child {node {}};
      \tikzstyle{level 1}=[counterclockwise from=120,level distance=9mm,sibling angle=120]
      \node (A2) [draw,circle,inner sep=1] at (180:1) {$\scriptstyle{0}$} child {node {}} child {node {2}};
      \tikzstyle{level 1}=[counterclockwise from=270,level distance=9mm,sibling angle=120]
			\node (A3) [draw,circle,inner sep=1] at (270:1) {$\scriptstyle{2}$} child {node {1}};

			\path (A0) edge [bend left=0.000000] (A1);
      \path (A1) edge [bend left=0.000000] (A2);
      \path (A0) edge [bend left=0.000000] (A3);
	\end{tikzpicture}\right]
-\left[	\begin{tikzpicture}[baseline,el/.style = {inner sep=2pt, align=left, sloped},every child node/.style={inner sep=1,font=\tiny}]
      \path(0,0) ellipse (2 and 2);
      \tikzstyle{level 1}=[counterclockwise from=0,level distance=9mm,sibling angle=120]
			\node (A0) [draw,circle,inner sep=1] at (0:1) {$\scriptstyle{0}$} child {node {}};
      \tikzstyle{level 1}=[counterclockwise from=90,level distance=9mm,sibling angle=120]
			\node (A1) [draw,circle,inner sep=1] at (90:1) {$\scriptstyle{0}$} child {node {}};
      \tikzstyle{level 1}=[counterclockwise from=120,level distance=9mm,sibling angle=120]
      \node (A2) [draw,circle,inner sep=1] at (180:1) {$\scriptstyle{0}$} child {node {1}} child {node {2}};
      \tikzstyle{level 1}=[counterclockwise from=270,level distance=9mm,sibling angle=120]
			\node (A3) [draw,circle,inner sep=1] at (270:1) {$\scriptstyle{2}$} child {node {}};

			\path (A0) edge [bend left=0.000000] (A1);
      \path (A1) edge [bend left=0.000000] (A2);
      \path (A0) edge [bend left=0.000000] (A3);
	\end{tikzpicture}\right]
\]\[
+\frac{1}{3}\left[	\begin{tikzpicture}[baseline,el/.style = {inner sep=2pt, align=left, sloped},every child node/.style={inner sep=1,font=\tiny}]
      \path(0,0) ellipse (2 and 2);
      \tikzstyle{level 1}=[counterclockwise from=0,level distance=9mm,sibling angle=120]
			\node (A0) [draw,circle,inner sep=1] at (0:1) {$\scriptstyle{0}$} child {node {}};
      \tikzstyle{level 1}=[counterclockwise from=45,level distance=9mm,sibling angle=120]
			\node (A1) [draw,circle,inner sep=1] at (90:1) {$\scriptstyle{0}$} child {node {}} child {node {}};
      \tikzstyle{level 1}=[counterclockwise from=120,level distance=9mm,sibling angle=120]
      \node (A2) [draw,circle,inner sep=1] at (180:1) {$\scriptstyle{0}$} child {node {1}} child {node {2}};
			\node (A3) [draw,circle,inner sep=1] at (270:1) {$\scriptstyle{2}$};

			\path (A0) edge [bend left=0.000000] (A1);
      \path (A3) edge [bend left=0.000000] (A2);
      \path (A0) edge [bend left=0.000000] (A3);
	\end{tikzpicture}\right]
-\left[	\begin{tikzpicture}[baseline,el/.style = {inner sep=2pt, align=left, sloped},every child node/.style={inner sep=1,font=\tiny}]
      \path(0,0) ellipse (2 and 2);
      \tikzstyle{level 1}=[counterclockwise from=0,level distance=9mm,sibling angle=120]
			\node (A0) [draw,circle,inner sep=1] at (0:1) {$\scriptstyle{0}$} child {node {}};
      \tikzstyle{level 1}=[counterclockwise from=45,level distance=9mm,sibling angle=120]
			\node (A1) [draw,circle,inner sep=1] at (90:1) {$\scriptstyle{0}$} child {node {1}} child {node {}};
      \tikzstyle{level 1}=[counterclockwise from=120,level distance=9mm,sibling angle=120]
      \node (A2) [draw,circle,inner sep=1] at (180:1) {$\scriptstyle{0}$} child {node {}} child {node {2}};
			\node (A3) [draw,circle,inner sep=1] at (270:1) {$\scriptstyle{2}$};

			\path (A0) edge [bend left=0.000000] (A1);
      \path (A3) edge [bend left=0.000000] (A2);
      \path (A0) edge [bend left=0.000000] (A3);
	\end{tikzpicture}\right]
+2\left[	\begin{tikzpicture}[baseline,el/.style = {inner sep=2pt, align=left, sloped},every child node/.style={inner sep=1,font=\tiny}]
      \path(0,0) ellipse (2 and 2);
      \tikzstyle{level 1}=[counterclockwise from=0,level distance=9mm,sibling angle=120]
			\node (A0) [draw,circle,inner sep=1] at (0:1) {$\scriptstyle{0}$} child {node {}};
      \tikzstyle{level 1}=[counterclockwise from=45,level distance=9mm,sibling angle=120]
			\node (A1) [draw,circle,inner sep=1] at (90:1) {$\scriptstyle{0}$} child {node {2}} child {node {}};
      \tikzstyle{level 1}=[counterclockwise from=120,level distance=9mm,sibling angle=120]
      \node (A2) [draw,circle,inner sep=1] at (180:1) {$\scriptstyle{0}$} child {node {}} child {node {1}};
			\node (A3) [draw,circle,inner sep=1] at (270:1) {$\scriptstyle{2}$};

			\path (A0) edge [bend left=0.000000] (A1);
      \path (A3) edge [bend left=0.000000] (A2);
      \path (A0) edge [bend left=0.000000] (A3);
	\end{tikzpicture}\right]
-3\left[	\begin{tikzpicture}[baseline,el/.style = {inner sep=2pt, align=left, sloped},every child node/.style={inner sep=1,font=\tiny}]
      \path(0,0) ellipse (2 and 2);
      \tikzstyle{level 1}=[counterclockwise from=0,level distance=9mm,sibling angle=120]
			\node (A0) [draw,circle,inner sep=1] at (0:1) {$\scriptstyle{0}$} child {node {}};
      \tikzstyle{level 1}=[counterclockwise from=45,level distance=9mm,sibling angle=120]
			\node (A1) [draw,circle,inner sep=1] at (90:1) {$\scriptstyle{0}$} child {node {1}} child {node {2}};
      \tikzstyle{level 1}=[counterclockwise from=120,level distance=9mm,sibling angle=120]
      \node (A2) [draw,circle,inner sep=1] at (180:1) {$\scriptstyle{0}$} child {node {}} child {node {}};
			\node (A3) [draw,circle,inner sep=1] at (270:1) {$\scriptstyle{2}$};

			\path (A0) edge [bend left=0.000000] (A1);
      \path (A3) edge [bend left=0.000000] (A2);
      \path (A0) edge [bend left=0.000000] (A3);
	\end{tikzpicture}\right].
\]
We define
	\[f(i):=\left<\psi\omega,a\omega^i,b\omega^{l-i}\right>-2\left<a\omega^{i+1},b\omega^{l-i}\right>-2\left<a\omega^i,b\omega^{l-i+1}\right>,\]
	\[h:=\left<a b\omega^{l-1},\omega^{2}\right>-\frac{1}{3}\left<a b\omega^{l-2},\omega^{3}\right>,\]
in order to obtain
	\[\Sigma(L_{2,5})(a\omega^k,b\omega^{l-2-k},\omega,\omega,\omega) \sim f(k+2)-2 f(k+1)+f(k)-h,\]
	for $0\leq k\leq l$.
	Applying Lemma~\ref{diffeq} gives the desired formula
\end{proof}
\begin{lemma}
	\label{nonprimitivereconstruction}
	Let $a,b,c\in H^{\leq2}(X)$, the following holds for all $i,j,k\geq0$
		\[\left<\psi a\omega^i,b\omega^j,c\omega^k\right>\sim \text{\emph{a linear combination of 2 pointed primary invariants}}.\]
\end{lemma}
\begin{proof}
	Let us define the equivalence relation $\sim\sim$ by saying that a linear combination of Gromov-Witten invariants is equivalent to zero if it is equivalent to zero for the $\sim$ equivalence relation or if it is a linear combination of 2-pointed primary invariants.

	By Lemma~\ref{psiomega} we have
\[\left<\psi\omega,a\omega^i,b\omega^{j}\right>\sim\sim 0.\]
	Now pulling back Corollary~\ref{psiomega2} and inserting $\gamma$ gives
	\[\left<\psi\omega^{2},b\omega^{j},\gamma\right>\sim\sim0.\]
	By the hypothesis on $X$ we can rewrite $ab\omega^{i+j-2}$ as $a'\omega^{k'}$ for some $a'\in H^{\leq2}(X)$, $k'\in\mathbb{Z}_{\geq0}$.
	So we can apply the above formulas to the pullback of Lemma~\ref{243lemma} (where we insert $\gamma$) to obtain
	\[(i+j)\left<\psi a\omega^i,b\omega^j,\gamma\right>\sim\sim j\left<\psi a,b\omega^{i+j},\gamma\right>.\]
	Repeatedly applying this formula gives
	\begin{equation}
		\label{psia}
	\left<\psi a\omega^{i},b\omega^{j},c\omega^k\right>\sim\sim jk\left<\psi a\omega^{i+j+k-2},b\omega,c\omega\right>.
	\end{equation}
	So it is sufficient to proof that
	\[\left<\psi a\omega^i,b\omega,c\omega\right>\sim\sim0\]
	We apply (\ref{psia}) to $L_{2,5}$, the relation from the proof of Lemma~\ref{psiomega}, to obtain
	\begin{align*}
		(-1)^{|a|(|b|+|c|)}\Sigma(L_{2,5})(b\omega^i,c,a,\omega,\omega) \sim\sim & \left<\psi a,b\omega^{i},c\omega^2\right> -2\left<\psi a,b\omega^{i+1},c\omega\right> \\
	\sim\sim &
	(2i-2(i+1))\left<\psi a\omega^i,b\omega,c\omega\right>.
\end{align*}
\end{proof}

What is left is to find an expression for primary invariants.
	Using our computer program we find a symmetric relation $L_{2,6}\in R^3_{2,6}$ that has the following primitive part:
\[\left[
			\begin{tikzpicture}[baseline,every node/.style={draw,circle,inner sep=1}]
      \path(0,0) ellipse (2 and 2);
      \node (A0) at (0:1) {$\scriptstyle{0}$};
      \tikzstyle{level 1}=[counterclockwise from=30,level distance=9mm,sibling angle=120]
      \node (A1) at (90:1) {$\scriptstyle{0}$} child child;
      \tikzstyle{level 1}=[counterclockwise from=120,level distance=9mm,sibling angle=120]
      \node (A2) at (180:1) {$\scriptstyle{0}$} child child;
      \tikzstyle{level 1}=[counterclockwise from=210,level distance=9mm,sibling angle=120]
      \node (A3) at (270:1) {$\scriptstyle{2}$} child child;

      \path (A0) edge [bend left=0.000000] (A1);
      \path (A0) edge [bend left=0.000000] (A2);
      \path (A0) edge [bend left=0.000000] (A3);
    \end{tikzpicture}
\right]-2\left[
			\begin{tikzpicture}[baseline,every node/.style={draw,circle,inner sep=1}]
      \path(0,0) ellipse (2 and 2);
      \tikzstyle{level 1}=[counterclockwise from=0,level distance=9mm,sibling angle=0]
      \node (A0) at (0:1) {$\scriptstyle{0}$} child;
      \tikzstyle{level 1}=[counterclockwise from=30,level distance=9mm,sibling angle=120]
      \node (A1) at (90:1) {$\scriptstyle{0}$} child child;
      \tikzstyle{level 1}=[counterclockwise from=120,level distance=9mm,sibling angle=120]
      \node (A2) at (180:1) {$\scriptstyle{0}$} child child;
      \tikzstyle{level 1}=[counterclockwise from=270,level distance=9mm,sibling angle=0]
      \node (A3) at (270:1) {$\scriptstyle{2}$} child;

      \path (A0) edge [bend left=0.000000] (A1);
      \path (A0) edge [bend left=0.000000] (A3);
      \path (A2) edge [bend left=0.000000] (A3);
    \end{tikzpicture}
\right]+\left[
			\begin{tikzpicture}[baseline,every node/.style={draw,circle,inner sep=1}]
      \path(0,0) ellipse (2 and 2);
      \tikzstyle{level 1}=[counterclockwise from=-60,level distance=9mm,sibling angle=120]
      \node (A0) at (0:1) {$\scriptstyle{0}$} child child;
      \tikzstyle{level 1}=[counterclockwise from=30,level distance=9mm,sibling angle=120]
      \node (A1) at (90:1) {$\scriptstyle{0}$} child child;
      \tikzstyle{level 1}=[counterclockwise from=120,level distance=9mm,sibling angle=120]
      \node (A2) at (180:1) {$\scriptstyle{0}$} child child;
      \node (A3) at (270:1) {$\scriptstyle{2}$};

      \path (A0) edge [bend left=0.000000] (A3);
      \path (A1) edge [bend left=0.000000] (A3);
      \path (A2) edge [bend left=0.000000] (A3);
    \end{tikzpicture}
\right].\]
We write \gls{Phi} for the system of equations
\begin{equation}
	\label{Phidef}
\Big\{\Sigma(L_{2,6})(\gamma_1\omega^{k_1},\gamma_2\omega^{k_2},\gamma_3\omega^{k_3},\omega^{k_4+1},\omega^{k_5+1},\omega^{k_6+1})\sim0\Big\}_{k_1+k_2+k_3+k_4+k_5+k_6=k}.
\end{equation}
	Since the degree $k$ will be split over 3 points rather than 2, we can no longer use a simple difference equation to find a general expression.
	We have an infinite series of matrices $\Phi_k(a,b,c)$ for $a,b,c\in H^{\leq2}(X)$.
	However, we can reduce to the case where all of the degree is concentrated in the first point, i.e. invariants of the form
	\begin{equation*}
		\left<a  \omega^i,b  \omega,c\omega\right>
	\end{equation*}
	for $i\geq0$.
	This means we will only need to consider $\Phi_k(a\omega^j,b,c)$ where $j\geq0$ is an unspecified variable and $k$ is small.

	We take Equation (\ref{psiomega2}) and divide it by $(j+1)(j+2)$.
	We also pull it back and insert the cohomology class $b\omega^k$.
	Finally we apply Lemma~\ref{psiomega} to obtain a sum of primary invariants on the right-hand side.
	\begin{align*}
	\left<\psi \omega^{2},a\omega^{j},b\omega^k\right>
	\sim &
	\left<a\omega^{j},b\omega^{k+2}\right>
	-\frac{3j^2+3j-2}{(j+1)(j+2)}\left<a \omega^{j+2},b\omega^k\right>
	\\ &
	+\frac{2j}{j+1}\left<\psi\omega,a \omega^{j+1},b\omega^k\right>
	-\frac{2j}{j+1}\left<a \omega^{j+1},b\omega^{k+1}\right>
	\\ \sim &
	\left<a\omega^{j},b\omega^{k+2}\right>
	+\frac{j^2+5j+2}{(j+1)(j+2)}\left<a \omega^{j+2},b\omega^k\right>
	\\ &
	+\frac{2j}{j+1}\left<a \omega^{j+1},b\omega^{k+1}\right>
	-\frac{2j}{j+k+1}\left<a\omega^{j+k+1},b\omega\right>
	\\ &
	-\frac{2jk}{(j+1)(j+k+1)}\left<a\omega,b\omega^{j+k+1}\right>
		\\ &
		-\frac{jk}{3}\Big(
			3\left<a b\omega^{j+k},\omega^{2}\right>-\left<a b\omega^{j+k-1},\omega^{3}\right>
	\Big)
	.
	\end{align*}
	We swap the 2nd and 3rd point (i.e. we swap $a\omega^i$ and $b\omega^j$) and subtract the resulting formula.
	\begin{align}
		\label{Psirel}
		\begin{split}
		0  \sim&
	-\frac{2k}{(k+1)(k+2)}\left<a\omega^{j},b\omega^{k+2}\right>
	+\frac{2j}{(j+1)(j+2)}\left<a\omega^{j+2},b\omega^{k}\right>
	\\ &
	+\frac{2(j-k)}{(j+1)(k+1)}\left<a \omega^{j+1},b\omega^{k+1}\right>
	-\frac{2j}{(k+1)(j+k+1)}\left<a\omega^{j+k+1},b\omega\right>
	\\ &
	+\frac{2k}{(j+1)(j+k+1)}\left<a\omega,b\omega^{j+k+1}\right>
	.
	\end{split}
	\end{align}
	Let $\Psi_{j,k}(a,b)$ denote the right-hand side of (\ref{Psirel}).
	We define
	\[\text{\gls{Theta}}:=(j+2)(k+1)\Big((j+1)\Psi_{j,k}(a,b)-\frac{k(j+2)}{(k-1)}\Psi_{j+1,k-1}(a,b)\Big).\]
	We have
	\begin{align*}
		\Theta_{j,k}(a,b)  =&
		-\frac{2k(j+1)(j+2)}{k+2}\left<a\omega^{j},b\omega^{k+2}\right>
										 \\ &
				 +\frac{2(jk-k^2+k+2)(j+2)}{k-1}\left<a \omega^{j+1},b\omega^{k+1}\right>
	\\ &
	-\frac{2(j^2-2jk+3j-2k)(k+1)}{k-1}\left<a\omega^{j+2},b\omega^{k}\right>
	\\ &
	-\frac{2k(j+1)(j+2)(k+1)}{(k-1)(j+3)}\left<a\omega^{j+3},b\omega^{k-1}\right>
	\\ &
	+\frac{4(j+1)(j+2)}{k-1}\left<a\omega^{j+k+2},b\omega\right>
	.
	\end{align*}

\begin{lemma}
	\label{primitivereconstruction}
	Let $a,b,c\in H^{\leq2}(X)$, the following holds for all $i,j,k\geq0$
	\begin{equation*}
		\left<a  \omega^i,b  \omega^j,c\omega^k\right>\sim0
	\end{equation*}
\end{lemma}
\begin{proof}
	Using $\Theta$ we can always reduce to the case $j,k\leq 2$.

	We write $\Theta_{j,k}(a,b)\left<\gamma\right>$ to denote the pullback of $\Theta_{j,k}(a,b)$ where we insert the cohomology class $\gamma$.

The rest of the proof consists of proving that certain systems of equations are full rank.
We only list the equations here.
The calculation of these symbols and checking if the systems are full rank is done by our computer program.
A description of the outputs of the computer program can be found at the web adress given in the introduction.

	The relations
	\[\Phi_1(a\omega^i,\omega,\omega),\quad
	\Theta_{1,2}(1,\omega)\left<a\omega^i\right>,\quad\Theta_{i,2}(a,1)\left<\omega^2\right>.\]
	make up a full rank system of 4 unique equations in 4 variables, which proves that
	\[\left<a\omega^{i+2},\omega^2,\omega^2\right>\sim\left<a\omega^{i+1},\omega^3,\omega^2\right>\sim\left<a\omega^{i},\omega^3,\omega^3\right>\sim0,\]
	for $i\geq0$.
	Coupled with the one relation in $\Phi_0(a,\omega,\omega)$,
	\[\Sigma(L_{2,6})(a,\omega,\omega,\omega,\omega,\omega)\sim60\left<a\omega,\omega^2,\omega^2\right>\sim0,\]
	this proves that $\left<\gamma,\omega^i,\omega^j\right>\sim0$ for all $\gamma\in H^*(X)$, $i,j\geq0$.
	For the remainder of this proof we will set all invariants of this form to zero.

	Consider the equations
	\[\Phi_0(a,b,\omega),\quad\Phi_1(a,b,\omega),\quad\Phi_2(a\omega^i,b,\omega),\]
	\[(-1)^{|a||b|}\Theta_{1,2}(\omega,b)\left<a\omega^i\right>,\quad
	\Theta_{i+1,2}(a,1)\left<b\omega\right>,\quad\Theta_{i,2}(a,1)\left<b\omega^2\right>.\]
	Together these equations imply
	$\left<a\omega^{i},b\omega^j,\omega^k\right>\sim0,$
	for $i\geq0$, $0\leq j\leq 2$, and $0\leq k\leq3$.
	this proves that $\left<\gamma_1,\gamma_2,\omega^i\right>\sim0$ for all $\gamma_1,\gamma_2\in H^*(X)$, $i\geq0$.
	For the remainder of this proof we will set all invariants of this form to zero.

	Consider the equations
	\[\Phi_0(a,b,c),\quad\Phi_1(a,b,c),\quad\Phi_2(a\omega^i,b,c),\]
	\[\Theta_{i,2}(a,b)\left<c\omega\right>,\quad
	(-1)^{|b||c|}\Theta_{i,2}(a,c)\left<b\omega\right>.\]
	Together these equations imply
	$\left<a\omega^{i},b\omega^j,c\omega^k\right>\sim0,$
	for $i\geq0$, $0\leq j\leq 2$, and $0\leq k\leq2$.
	\end{proof}

\subsection{Tautological relations in genus 2}
In this chapter we describe how the relations we use can be expressed in terms of \emph{new} relations.
Many reconstructions for Gromov-Witten invariants such as those by Getzler and Belorousski-Pandharipande were the result of a \emph{new} relation.
The notion of a new relation was formally introduced by Pixton in \cite{pixtonconjrels} as follows.
\begin{definition}
	Let $R^\mathrm{old}$ consist of $S_{g,n}^{>0}\cdot R_{g',n'}$, $q_*(R_{g,n}\otimes R_{g',n'})$, and $\pi^*(R_{g,n})$ for all natural gluing and forgetful maps $q$ and $\pi$ and any $g,n,g',n'\in\mathbb{Z}_{\geq0}$.
	We define $\text{\gls{newrels}}:=R_{g,n}/(R_{g,n}\cap R^\mathrm{old})$ to be the space of \emph{new} relations.
\end{definition}

From Proposition 2 in \cite{pixtonconjrels} we obtain the following.
\begin{lemma}
	Let $g>0$ and let $[R]\in P_{g,n}\cap R^\mathrm{new}_{g,n}$, then $[R]$ has a representative $R'\in P_{g,n}$ that is symmetric, i.e. $R'$ is fixed by the action of $\mathfrak{S}_n$ permuting the points.
\end{lemma}

There are new relations in genus two that express $\kappa$-classes and degree two monomials $\psi$-classes in terms of decorated strata classes with at most one $\psi$-class.
Besides these relations the only known new relations are in $R_{2,3}^2$ and $R_{2,6}^3$.
The new relation in $R_{2,3}^2$ was first discovered by Belorousski and Pandharipande in \cite{belpandg2}.
The relation $L_\mathrm{BP}$ that they present is the unique representative (up to scalars) of the new relation in $R_{2,3}^2$.

We use the 3 relations $L_{2,i}$ for $4\leq i\leq 6$ in our proof of \ref{g2reconstruction}.
We have
\[L_\mathrm{BP}={\pi_2}_* L_{2,4},\]
and the relation $L_{2,6}$ represents the new relation in $R_{2,6}^3$.
The relations $L_{2,4}$ and $L_{2,5}$ are not new and we can recover them from $L_\mathrm{BP}$.

	For any $0\leq i\leq 4$, the relation $\pi_1^*(L_\mathrm{BP})\cdot\psi_i$ has terms with genus~$2$ components that have a degree 2 monomial in $\psi$-classes.
	We can express these terms using decorated strata classes that have at most one $\psi$-class.
	Modulo these simplifications we have for the primitive part,
	and also modulo pullbacks along gluing morphisms of the new relation in $R^1(\Mbstack_{0,4})$ we have
	\[2L_{2,4}=	\pi_1^*(L_\mathrm{BP})(-\psi_1+3\psi_2+\psi_3+\psi_4).\]
		Similarly, for $L_{2,5}$ we have
	\[6L_{2,5}=
			\pi_{1,2}^*(L_\mathrm{BP})(3\psi_1-4\psi_2-9\psi_4)
	+		\pi_{1,4}^*(L_\mathrm{BP})\cdot(-7\psi_1+28\psi_2+4\psi_4-16\psi_5)
\]\[
	+		\pi_{2,4}^*(L_\mathrm{BP})\cdot(-3\psi_4+12\psi_5)
+3	\pi_{4,5}^*(L_\mathrm{BP})\cdot\psi_4
.\]
Note that these two formulas give specific representatives in equivalence classes of relations that have the same symbol.
\begin{remark}
	The restriction of the relation $L_{2,6}$ to the moduli space of stable curves with only rational tails was found by Tavakol in \cite{tavakolg2}.
	For our purposes, it is important to know that this relation extends to a \emph{tautological} relation on all of $\Mbstack_{2,6}$.
	This follows immediately when we obtain it as a Pixton relations.
\end{remark}

\section{Genus 2 Gromov-Witten invariants of blowups of the projective plane}

In this chapter we calculate descendant Gromov-Witten invariants of \gls{Xr}, the blowup of $\mathbb{P}^2$ at points $p_1,\ldots,p_r$ in general position.
Although we do not use Theorem \ref{g2reconstruction}, we do use the relation $L_{2,4}\in R^3_{2,4}$ from the proof of the theorem.
\begin{theorem}
	\label{p2blowupcomputation}
	The algorithm described in Sections \ref{dpg0rec}, \ref{dpg1rec}, and \ref{dpg2rec} reconstructs all genus~$0$,$1$, and $2$ Gromov-Witten invariants with descendants of $X_r$ from the finitely many initial cases in Lemma~\ref{DPbasecases} and $<pt^2>^{X_r}_{0,H}=1$.
\end{theorem}

	A cohomology basis for $X_r$ is given by $1,pt,H,E_1,\ldots,E_r$, where $H$ is the hyperplane class and $E_i$ is the class of the exceptional divisor at the $i$-th blown up point.
	We have
	\[H\cdot H = 1,\qquad H\cdot E_i=0,\qquad E_i\cdot E_j=\delta_{i,j},\]
	for any $1\leq i,j\leq r$.

\begin{lemma}
	\label{effectivedp}
	Let $C$ be an effective curve on $X_r$ of class $\beta$.
	Then either $\beta$ is a multiple of $E_i$ for some $1\leq i \leq r$,
	or $\beta=dH-\sum_{i=1}^{r}\alpha_iE_i$ with $0\leq\alpha_i\leq d$.
	In the second case the coefficient $a_i$ is equal to the multiplicity of $\pi(C)$ at $p_i$,
	where $\pi:X_r\rightarrow\mathbb{P}^2$ is the blow-down morphism.
\end{lemma}
\begin{proof}
	This is a special case of Lemma~2.3 in \cite{coskuneffective}.
\end{proof}

The anticanonical divisor of $X_r$ is given by $-K_{X_r}=3d-\sum_{i=1}^r E_i$ so the virtual dimension of $\Mbstack_{g,n}(X_r,\beta)$ is
\[\int_{dH-\sum_{i=1}^{r}\alpha_iE_i} -K_{X_r}+(\text{dim}(X)-3)(1-g)+n=
	3d-\sum_{i=1}^{r}\alpha_i+g-1+n
.\]

Looking at the virtual dimension, we see that all nonzero invariants where $d\neq0$ are of the form \glsuseri{NHPK}
	\begin{align*}
		N^{(g)}_{d,\alpha} & := \left<pt^{3d-\sum \alpha_i + g - 1} \right>_{g,dH-\sum\alpha_i E_i},\\
		H^{(2)}_{d,\alpha} & := \left<\tau_1(H)\cdot pt^{3d-\sum \alpha_i} \right>_{2,dH-\sum\alpha_i E_i},\\
		P^{(2)}_{d,\alpha} & := \left<\tau_1(pt)\cdot pt^{3d-\sum \alpha_i - 1} \right>_{2,dH-\sum\alpha_i E_i},\\
	  K^{(2)}_{d,\alpha} & := \left<\tau_1(E_1)\cdot pt^{3d-\sum \alpha_i} \right>_{2,dH-\sum\alpha_i E_i}.
	\end{align*}
	By the local nature of blowing-up we have
\[N^{(g)}_{d,(\alpha_1,\ldots,\alpha_{r-1},0)}=N^{(g)}_{d,(\alpha_1,\ldots,\alpha_{r-1})}.\]
	A similar equality holds for $H^{(2)}_{d,\alpha}$ and $P^{(2)}_{d,\alpha}$.
	It holds for $K^{(2)}_{d,\alpha}$ only when $r>1$.

\begin{lemma}
	\label{DPbasecases}
	For genus up to 2, the only nontrivial invariants with $d=0$ are
	\[N^{(0)}_{0,E_i}=1, \quad \left<H\right>_{1,0}=-\frac{1}{8}, \quad \left<E_i\right>_{1,0}=-\frac{1}{24},\]\[ \quad H^{(2)}_{0,0}=-\frac{1}{960},\quad K^{(2)}_{0,0}=-\frac{1}{2880}.\]
\end{lemma}
\begin{proof}
	By looking at the virtual dimension it follows that (up to linear combinations) these are the only possible nonzero invariants.
	The exceptional divisor $E_i$ is itself a rigid curve of genus~$0$.
	The genus~$1$ cases follow from Lemma~\ref{genstring}.

For genus~$2$ we have $\Mbstack_{2,1}(X_r,0)=\Mbstack_{2,1}\times X_r$ and by \RN{6}.6.3 in \cite{maninbookfrobenius}
\begin{equation}
	\label{virtualclasscomputation}
\left[\Mbstack_{2,1}\times X_r\right]^\mathrm{vir}=c_4(\mathbb{E}^\vee\boxtimes TX_r),
\end{equation}
	where $\mathbb{E}$ is the Hodge bundle and $TX_r$ is the tangent bundle of $X_r$.
	This gives
	\[	\int_{\left[\Mbstack_{2,1}\times X_r\right]^\mathrm{vir}}\psi\boxtimes D,
=
\int_{\Mbstack_{2,1}}\psi(\lambda_1^3-3\lambda_1\lambda_2)\int_{X_r}-K_{X_r}D,
\]
which equals $-\frac{1}{960}$ when $D=H$ and $-\frac{1}{2880}$ when $D=E_i$.
\end{proof}
	So our recursive strategy will rely on computing invariants modulo invariants with lower $d$.

	Let $L\in R_{2,n}$ be a relation.
	For fixed $\beta\in H_2(X_r)$, we define
	\[L(\gamma_1,\ldots,\gamma_n):=<\pi_{\{n+1,\ldots,n+m\}}^*(L);\gamma_1,\ldots,\gamma_n,pt,pt,\ldots,pt>_\beta,\]
	where for $m$, which is the number of points to add, we take the unique choice such that the resulting invariants are not all zero for dimensional reasons.

	We write $\alpha+[j]$ for the tuple $\alpha'$ where $\alpha'_i=\alpha_i$ for $i\neq j$ and $\alpha'_j=\alpha_j+1$.

\begin{definition}
We define an equivalence relation $\approx$, where a linear combination of Gromov-Witten invariants is equivalent to zero if it can be expressed in terms of invariants with lower $d$ or lower genus.
\end{definition}

\subsection{Reconstructing genus~$0$ Gromov-Witten invariants}
\label{dpg0rec}
	For the genus $0$ invariants we use the method described in \cite{gotpand}.
	They use the relation $L_{0,4}$ in $R^1(\Mbstack_{0,4})$ to get two recursive formulas:
	For $d-\sum\alpha_i-1 \geq 3$, $L_{0,4}(H,H,pt,pt)$ gives
	\[N^{(0)}_{d,\alpha}\approx 0,\]
	and when $d-\sum\alpha_i-1 \geq 1$, from $L_{0,4}(H,H,E_1,E_1)$ we obtain
	\[\alpha_1d^2 N^{(0)}_{d,\alpha+[1]}+(\alpha_1^2-d^2) N^{(0)}_{d,\alpha}\approx 0,\]
	When $d-\sum\alpha_i-1 = 1$ the invariant $N^{(0)}_{d,\alpha+[1]}$ has zero points.
	So from these two formulas and the base case $N^{(0)}_{1,0}=1$ we can recursively calculate all genus~$0$ invariants.

\subsection{Reconstructing genus~$1$ Gromov-Witten invariants}
\label{dpg1rec}
	For genus~$1$ we use Getzler's relation $L_{1,4}$ from \cite{getzlerg1}.
	For $3d-\sum\alpha_i \geq 2$, $L_{1,4}(H,H,H,H)$ gives
\[N^{(1)}_{d,\alpha}\approx 0,\]
and when $3d-\sum\alpha_i \geq 0$, from $L_{1,4}(E_1,E_1,E_1,E_1)$ we obtain
\[\frac{(\alpha_1+2)(\alpha_1+3)}{3}N^{(1)}_{d,\alpha+2[1]}+\frac{2\alpha_1}{3}N^{(1)}_{d,\alpha+[1]}-N^{(1)}_{d,\alpha}\approx 0.\]
Together with the genus~$0$ invariants this is sufficient to recursively calculate all genus~$1$ invariants.

\subsection{Reconstructing genus~$2$ Gromov-Witten invariants}
\label{dpg2rec}
In \cite{belpandg2} the genus~$2$ Gromov-Witten invariants of $\mathbb{P}^2$ are calculated using the relation $L_\mathrm{BP}$.
So we can assume that $r>0$ and we always have an exeptional divisor $E_1$.

	We calculate the invariants of type $H$ using $L_{2,4}(H,E_1,E_1,E_1)$, which gives
	\[-(\alpha_1+2)H^{(2)}_{d,\alpha+[2]}+H^{(2)}_{d,\alpha+[1]}\approx 0.\]
	for $3d-\sum\alpha_i + 1 \geq 3$.

	For type $P$ we use $L_\mathrm{BP}(H,H,H)$ to obtain
	\[dP^{(2)}_{d,\alpha}-H^{(2)}_{d,\alpha}\approx 0.\]
	for $3d-\sum\alpha_i + 1 \geq 2$.

	For type $N$ we use the linear combination
	\[\frac{3d}{2}L_\mathrm{BP}(E_1,H,H) -L_\mathrm{BP}(E_1,E_1,H) +\frac{d}{2}L_\mathrm{BP}(H,H,H),\]
	which gives the formula
	\[-(\alpha_1+2)dN^{(2)}_{d,\alpha+2[1]}+dN^{(2)}_{d,\alpha+[1]}-dP^{(2)}_{d,\alpha}+H^{(2)}_{d,\alpha}-(\alpha_1+1)H^{(2)}_{d,\alpha+[1]}\approx 0,\]
	for $3d-\sum\alpha_i + 1 \geq 2$.

	Finally for type $K$ we take $L_\mathrm{BP}(E_1,E_1,E_1)$, which gives
	\[-2(\alpha_1+2)N^{(2)}_{d,\alpha+2[1]}+\frac{2-3(\alpha_1+1)}{2}N^{(2)}_{d,\alpha+[1]}+\frac{3\alpha_1}{2}P^{(2)}_{d,\alpha}-\frac{3}{2}K^{(2)}_{d,\alpha}-K^{(2)}_{d,\alpha+[1]}\approx 0.\]
	for $3d-\sum\alpha_i + 1 \geq 2$.

\subsection{Optimization and some numerical results}
	We use similar optimizations as in \cite{gotpand}:
	The Cremona transformation and the fact that when $3d-\sum_{i=1}^{r-1}\alpha_i+1\geq0$, we have
\begin{equation}
	\label{removeonexr}
	N^{(0)}_{d,(\alpha_1,\ldots,\alpha_{r-1},1)}=N^{(0)}_{d,(\alpha_1,\ldots,\alpha_{r-1})}.
\end{equation}
	I do not know a proof for this when $g>0$ and $r>8$.
	However for all the numbers we calculated it does hold.
	The same also holds for $H^{(2)}_{d,\alpha}$ and $P^{(2)}_{d,\alpha}$.
	It holds up in our calculations for $K^{(2)}_{d,\alpha}$ when $r>1$.

	Assuming this property, blown up points for which $\alpha_i=0,1$ do not contribute, so the easiest interesting invariants to compute are the ones for which $\alpha=2^r=(2,\ldots,2)$.
We conclude by explicitly listing some such invariants and some descendant invariants:
{\small
	\begin{align*}
	& K^{(2)}_{3,1} = \;-\frac{1}{12}, 	&&N^{(2)}_{6,2^8} = \;1, 	&& N^{(1)}_{8,2^{11}} = \;24949650, \\
	& P^{(2)}_{4,2} = \;-\frac{2}{3}, 	&&N^{(2)}_{6,2^9} = \;0, 	&& N^{(1)}_{8,2^{12}} = \;\frac{10527885}{2}, \\
	& H^{(2)}_{4,2} = \;-\frac{5}{3}, 	&& N^{(2)}_{7,2^{10}} = \;3113, && N^{(1)}_{9,2^{12}} = \;58460483880,\\
	& H^{(2)}_{4,2^2} = \;-\frac{1}{3}, 	&& N^{(2)}_{7,2^{11}} = \;313, && N^{(1)}_{9,2^{13}} = \;14967968670,\\
	& H^{(2)}_{5,3} = \;72, && N^{(2)}_{8,2^{11}} = \;25721212, && N^{(1)}_{10,2^{15}} = \;12335169291480,\\
	& H^{(2)}_{6,2^4} = \;157689, && N^{(2)}_{8,2^{12}} = \;4604976, && N^{(1)}_{11,2^{16}} = \;44275609195448900.
	\end{align*}
}
We have also checked the numbers listed in \cite{vakilcountingcurves}, \cite{shovalshustin}, and \cite{brugalle}.
Our numbers agree except for the invariant $N^{(1)}_{5,2^2}$ which is known to be a mistake in \cite{vakilcountingcurves} (see \cite[Chapter 8]{chaudhuridas}, which gives the correct number).

Since the primary Gromov-Witten invariants are only enumerative when $r\leq 8$, they can be rational or negative when $r>8$.
Hence we have a non-integer invariant $N^{(1)}_{8,2^{12}}$.
It is interesting that other similar invariants we found are in fact all integers.

\printbibliography
\end{document}